\documentclass{amsart}
\usepackage{amsmath,amsthm,amsfonts,amssymb,amscd,amsbsy,dsfont,hyperref,graphicx,tabularx}
\usepackage[all]{xy}

\hypersetup{
    pdftoolbar=true,
    pdfmenubar=true,
    pdffitwindow=false,
    pdfstartview={FitH},
    pdftitle={Equivariant bifurcation in geometric variational problems},
    pdfauthor={R. G. Bettiol, P. Piccione and G. Siciliano},
    pdfsubject={},
    pdfkeywords={},
    pdfnewwindow=true,
    colorlinks=true, 
    linkcolor=blue,
    citecolor=blue,
    urlcolor=blue,
}

\author[R. G. Bettiol]{Renato G. Bettiol}
\author[P. Piccione]{Paolo Piccione}
\author[G. Siciliano]{Gaetano Siciliano}

\address{
\begin{tabular}{lll}
University of Notre Dame & &Universidade de S\~ao Paulo \\
Department of Mathematics & & Departamento de Matem\'atica \\
255 Hurley Building & & Rua do Mat\~ao, 1010 \\
Notre Dame, IN, 46556-4618, USA & & S\~ao Paulo, SP, 05508-090, Brazil\\
\emph{E-mail address}: {\tt rbettiol@nd.edu} & & \emph{E-mail address}: {\tt piccione@ime.usp.br}\\
 && \emph{E-mail address}: {\tt gaetano.siciliano@gmail.com}
\end{tabular}
}
\subjclass[2010]{58E07, 58E09, 46T05, 58D19, 53A10}
\date{July 31st, 2013} 

\theoremstyle{definition}\newtheorem*{defin*}{Definition}
\theoremstyle{plain}\newtheorem{teo}{Theorem}
\theoremstyle{plain}\newtheorem{prop}[teo]{Proposition}
\theoremstyle{plain}\newtheorem{lem}[teo]{Lemma}
\theoremstyle{plain}\newtheorem{cor}[teo]{Corollary}
\theoremstyle{definition}\newtheorem{defin}[teo]{Definition}
\theoremstyle{remark}\newtheorem{rem}[teo]{Remark}
\theoremstyle{plain}\newtheorem{example}[teo]{Example}

\newcommand{\dd}{\mathrm{d}}
\newcommand{\ev}{\mathrm{ev}}
\newcommand{\Iso}{\mathrm{Iso}}
\newcommand{\Hor}{\mathrm{Hor}}
\newcommand{\vol}{\mathrm{vol}}
\newcommand{\Vol}{\mathrm{Vol}}
\newcommand{\Area}{\mathrm{Area}}
\newcommand{\Diff}{\mathrm{Diff}}

\allowdisplaybreaks
\numberwithin{equation}{section}
\numberwithin{teo}{section}

\title[Equivariant bifurcation]{Equivariant bifurcation\\ in geometric variational problems}

\thanks{The first named author is supported by NSF, grant DMS-0941615, USA.
The second named author is partially supported by Fapesp and by CNPq, Brazil.
The third named author is supported by Fapesp, grant n.\ 2011/01081-9,  
CNPq grant 302523/2012-0, Brazil, and by M.I.U.R. - P.R.I.N. ``Metodi variazionali e topologici nello studio di fenomeni non lineari'', Italy.
}

\begin{document}

\begin{abstract}
We prove an extension of a celebrated equivariant bifurcation result of J. Smoller and A. Wasserman~\cite{SmoWas}, in an abstract framework for geometric variational problems. With this purpose, we prove a slice theorem for continuous affine actions of a (finite-dimensional) Lie group on Banach manifolds. As an application, we discuss equivariant bifurcation of constant mean curvature hypersurfaces, providing a few concrete examples and counter-examples.
\end{abstract}

\maketitle
\tableofcontents

\vspace{-2.5em}

\begin{section}{Introduction}

Most geometric variational problems are invariant under a symmetry group, in the sense that the geometric objects of interest are critical points of a functional invariant under the action of a Lie group. For example, the rotation action of $\mathds S^1$ on the space of loops of a Riemannian manifold $M$ leaves invariant the energy functional (whose critical points are closed geodesics on $M$). As a more interesting example, the action of the isometry group of a Riemannian manifold $\overline M$ leaves invariant the area functional (whose critical points with constrained volume are constant mean curvature (CMC) submanifolds $M\hookrightarrow\overline M$). The aim of this paper is to develop an abstract \emph{equivariant} bifurcation theory for families of critical points of variational problems as the above, tailored for geometric applications. In a certain sense, this problem is complementary to our study of an Implicit Function Theorem for such variational problems, see~\cite{BPS1}, namely characterizing when it fails.

Equivariant bifurcation for a $1$-parameter family of gradient-like operators invariant under the action of a Lie group on a Banach space was pioneered by the work of J. Smoller and A. Wasserman~\cite{SmoWas}. They found sufficient conditions for the existence of a bifurcation instant in a $1$-parameter family of zeros of such a path of operators, when these zeroes are fixed points of the action. The sufficient condition is stated in terms of the induced isotropy representations on the negative eigenspace of the linearized operators. This result was then successfully used to obtain bifurcation of radial solutions to semilinear elliptic PDEs in a disk with homogeneous linear boundary conditions, among other similar applications.

Nevertheless, in applications to geometric variational problems, it is too restrictive to assume that the starting $1$-parameter family of solutions is formed only by \emph{fixed points} of the action. Typically, variational problems involving maps with values in Riemannian manifolds are invariant under the isometry group of the target manifold, which acts by left-composition. It often is a natural situation that the given family of critical points is only invariant under a smaller group of isometries, i.e., the orbits of such points may not consist of single points, although they may also have non trivial isotropy. It is also important to observe that, in many cases, the action of the symmetry group is not everywhere differentiable. For instance, the (left-composition) action of the isometry group of $\overline M$ on the space of $\mathcal C^k$ \emph{unparameterized embeddings} of a compact manifold $M$ into $\overline M$ is only continuous, and differentiable only at $\mathcal C^\infty$ embeddings, see \cite{AliPic10}. This is the action one has to consider when studying the CMC variational problem. Finally, it is also common to have only a \emph{local} action of a symmetry group (which is also the case in the CMC variational problem).

In the present paper, we take into account all of the above observations and extend the classic equivariant bifurcation result of J. Smoller and A. Wasserman~\cite{SmoWas} to this more general situation. Let us describe with more details our main abstract bifurcation results, Theorems~\ref{thm:Gbifurcation} and \ref{thm:Gbifurcation2}. Assume $\mathcal M$ is a Banach manifold endowed with a connection and $G$ is a compact Lie group acting\footnote{To simplify our discussion, we suppose here that the action of $G$ is globally defined, although the results described in the sequel also hold for the more general case of \emph{local actions}.} continuously by affine diffeomorphisms on $\mathcal M$. Let $\mathfrak f_\lambda\colon \mathcal M\to\mathds R$ be a family of smooth $G$-invariant functionals, parameterized by $\lambda\in[a,b]$, and $\lambda\mapsto x_\lambda$ be a curve of critical points in $\mathcal M$, i.e., $\mathrm d\mathfrak f_\lambda(x_\lambda)=0$, for all $\lambda$. Under the appropriate Fredholmness assumptions on the second derivative of $\mathfrak f_\lambda$ at $x_\lambda$, we prove that if the following conditions are satisfied, there exists equivariant bifurcation at some $\lambda_*\in\left]a,b\right[\,$:
\begin{itemize}
\item \emph{Constant isotropy}: the isotropy group $H$ of $x_\lambda$ is a \emph{nice group}\footnote{e.g., this is satisfied if $H$ is a closed subgroup of $G$ with less than $5$ connected components, see Example~\ref{ex:nice}.} (in the sense of \cite{SmoWas}) and independent of $\lambda$;
\item \emph{Equivariant nondegeneracy}: the kernel of the second derivatives $\mathrm d^2\mathfrak f_a(x_a)$ and $\mathrm d^2\mathfrak f_b(x_b)$ coincides with the tangent space to the $G$-orbit of $x_a$ and of $x_b$, respectively;
\item \emph{Jump of negative isotropy representation}: the linear representations of $H$ on the ``negative eigenspaces'' of $\mathrm d^2f_a(x_a)$ and $\mathrm d^2f_b(x_b)$ are not equivalent.
\end{itemize}
In other words, the above three conditions imply that there exists a sequence $(x_n)_n$ in $\mathcal M$ and a sequence $(\lambda_n)_n$ in $[a,b]$, with $x_n\to x_{\lambda_*}$ and $\lambda_n\to\lambda_*$ as $n\to\infty$ such that for all $n$, $\mathrm d\mathfrak f_{\lambda_n}(x_n)=0$ and the orbit $G\cdot x_n$ is disjoint from the orbit $G\cdot x_{\lambda_n}$, see also Definition~\ref{def:eqbif}. A particular case of the third condition above is when there is a change of the Morse index (the sum of dimensions of the negative eigenspaces) from $x_a$ to $x_b$. Clearly, having the same dimension is a necessary condition for two representations to be equivalent, so a jump of the Morse index also determines existence of equivariant bifurcation.

The key idea for the proof of the above results is the construction of slices for group actions, and the reduction of the variational problem to a given slice (where a nonlinear formulation of the classic result of J. Smoller and A. Wasserman~\cite{SmoWas} can now be applied). Although slices for continuous group actions exist in a general topological setting (see \cite{Bredon}), when using variational calculus, one needs a stronger notion of slice (with some differentiability properties). Typically, differentiable slices are constructed applying the exponential map to the normal space of an orbit. This does not work in the general case of actions on Banach manifolds, that may not admit a (complete) invariant inner product. Our central observation is that, in a Banach manifold setting, a similar construction can be performed using the exponential of any \emph{invariant connection}, which exists naturally in many interesting situations. This exponential is then applied to some invariant closed complement of the tangent space to a differentiable group orbit. Invariant closed complements always exist in the case of strongly continuous group actions on Banach spaces (see Lemma~\ref{thm:preparatory}). Thus, the core of the present paper consists in a description of the main properties of connections on infinite-dimensional Banach manifolds (or Banach vector bundles), and the construction of smooth slices for continuous affine (local) actions.

As an example of application of this theory, we obtain bifurcation results for families of CMC hypersurfaces, see Theorems~\ref{thm:MorseCMCbif} and \ref{thm:equivCMCbifurcation}. Those are then applied to concrete families of Clifford tori in round and Berger spheres, and of rotationally symmetric surfaces in $\mathds R^3$. In those cases, a few recent bifurcation results by the second named author and others are reobtained, see~\cite{AliPic11,KoiPalPic2011,LimLimPic2011}. Other bifurcation results obtained by the first and second named authors for geometric variational problems with symmetries with a similar framework can be found in \cite{bp1,bp2,bp3}.

Some natural questions arise regarding further generalizations, e.g., when one considers the case in which the isotropy group of $x_\lambda$ depends on the parameter $\lambda$, described in Example~\ref{thm:remvariablegroup}. This is a topic of current research by the authors, as well as the study of other geometric applications, e.g., to the variational problem of constant anisotropic mean curvature hypersurfaces.

The paper is organized as follows. Section~\ref{sec:connection} contains general facts about connections on (infinite-dimensional) Banach vector bundles. Special emphasis is given to Banach bundles of sections of finite-dimensional vector bundles, endowed with an affine connection, over differentiable manifolds. This is the case of main interest in applications. In this context, the two main results (Proposition~\ref{thm:invdiffeos} and Corollary~\ref{thm:corrightcomp}) are that the map of right-composition with diffeomorphisms of the base manifolds is affine, as well as the map of left-composition with an affine map. Section~\ref{sec:slices} deals with the question of existence of slices for group actions on infinite-dimensional Banach manifolds. The main result of this section, Theorem~\ref{thm:existenceofslicecompact}, gives the existence of a slice through a point $x$ for affine actions, under the assumption of compactness of the isotropy of $x$. The case of local group actions is also discussed, see Subsection~\ref{sub:localactions}. Section~\ref{sec:mainbif} contains the main equivariant bifurcation results (Theorems~\ref{thm:Gbifurcation} and \ref{thm:Gbifurcation2}), which generalize \cite[Thm~2.1]{SmoWas} and \cite[Thm~3.3]{SmoWas} respectively. Section~\ref{sec:geomappl} contains a geometric application of the two abstract bifurcation results in the context of CMC embeddings, which was the original motivation for the development
of the theory. Concrete examples of bifurcation of CMC embeddings recently discovered are briefly presented in the end of this section. Finally, Appendix~\ref{app:A} describes the basic framework for the used nonlinear formulation of the results of J. Smoller and A. Wasserman~\cite{SmoWas}.
\end{section}

\begin{section}{Connections on infinite-dimensional manifolds}
\label{sec:connection}
We start by studying the notion of connection on a Banach vector bundle.
Given a connection on a finite-dimensional vector bundle $\pi^E\colon E\to M$ and a smooth
manifold $D$ (possibly with boundary), we describe the construction of a naturally
associated connection on the bundle $\pi^\mathcal E\colon \mathcal E\to\mathcal M$, where
$\mathcal E=\mathcal C^k(D,E)$, $\mathcal M=\mathcal C^k(D,M)$ and $\pi^\mathcal E$ is the
left composition with $\pi^E$. This is characterized as the unique connection for
which the evaluation maps $\ev_p\colon \mathcal E\to E$ are affine.
We show the invariance of this connection by the right action of the diffeomorphism
group of $D$. When $E=TM$ and the connection on $TM$ is the Levi-Civita connection
of some semi-Riemannian metric tensor $g$ on $M$,  then the associated connection
on $\mathcal C^k(D,M)$ is also invariant by the left action of the isometry group of $g$.
A classic reference on these topics is \cite{Eliasson}.

\subsection{Banach vector bundles}
Let $\mathcal M$ be a smooth Banach manifold, and $\pi^\mathcal E\colon \mathcal E\to\mathcal M$ be a smooth Banach
vector bundle on $\mathcal M$. This means that $\mathcal E=\bigcup_{x\in\mathcal M}\mathcal E_x$, with $\mathcal E_x=\pi^{-1}(x)$,
is a collection of vector spaces, and that it is given an \emph{atlas of compatible trivializations of $\mathcal E$}.
Given a Banach space $\mathcal E_0$, write
\[\mathrm{Fr}_{\mathcal E_0}(\mathcal E):=\bigcup_{x\in\mathcal M}\mathrm{Iso}(\mathcal E_0,\mathcal E_x),\]
where $\mathrm{Iso}(\mathcal E_0,\mathcal E_x)$ is the set of Banach space isomorphisms (bi-Lipschitz linear isomorphisms)
from $\mathcal E_0$ to $\mathcal E_x$.
For a vector bundle $\pi^\mathcal E\colon \mathcal E\to\mathcal M$ with fibers of finite dimension $n$, we will write $\mathrm{Fr}(\mathcal E)$ for $\mathrm{Fr}_{\mathds R^n}(\mathcal E)$.
A local trivialization of $\pi^\mathcal E\colon \mathcal E\to\mathcal M$ with domain
$\mathcal U\subset\mathcal M$ and typical fiber $\mathcal E_0$ is a local
section $s\colon \mathcal U\to\mathrm{Fr}_{\mathcal E_0}(\mathcal E)$. Two local trivializations $s_i$, with domain $\mathcal U_i\subset\mathcal M$
and typical fibers $\mathcal E_i$, $i=1,2$, are compatible if the transition map $s_2^{-1}s_1\colon \mathcal U_1\cap\mathcal U_2\to\mathrm{Iso}(\mathcal E_1,\mathcal E_2)$
is smooth. A collection $(\mathcal U_i,s_i,\mathcal E_i)_{i\in I}$ of local trivializations of $\mathcal E$ is an atlas if
the domains $\mathcal U_i$ cover $\mathcal M$. For details on the structure of such Banach vector bundles, see~\cite{PicTau2006}.

\subsection{Connections on Banach vector bundles}
\label{sub:connectionsBanachbundles}
A \emph{connection} on the Banach vector bundle $\pi^\mathcal E\colon \mathcal E\to\mathcal M$ is a smooth
map $P^\mathcal E\colon T\mathcal E\to\mathcal E$ such that:
\begin{itemize}
\item[(a)] for all $x\in\mathcal M$ and $e\in\mathcal E_x$, the restriction $P^\mathcal E_e=P^\mathcal E|_{T_e\mathcal E}$ is a linear map with values in $\mathcal E_x$;
\item[(b)] for any local trivialization $s\colon\mathcal U\to\mathrm{Fr}_{\mathcal E_0}(\mathcal E)$, there exists a smooth map
$$\mathcal U\ni x\longmapsto\Gamma_x\in\mathrm{Bil}(T_x\mathcal M\times\mathcal E_x,\mathcal E_x)$$ such that, denoting
by $\widetilde s\colon\mathcal E\vert_\mathcal U\to\mathcal E_0$ the map $\widetilde s(e)=s\big(\pi(e)\big)^{-1}(e)$, the following identity
holds for all $x\in\mathcal U$, $e\in\mathcal E_x$ and $\eta\in T_e\mathcal E$:
\[P^\mathcal E(\eta)=s(x)\big(\mathrm d\widetilde s_e(\eta)\big)+\Gamma_x\big(\mathrm d\pi^\mathcal E_e(\eta),e\big).\]
\end{itemize}
A standard argument shows that it suffices to have property (b) satisfied only for the set of local trivializations of an atlas.

A connection $\mathcal P^\mathcal E$ defines a distribution $\mathrm{Hor}(P^\mathcal E)$ on the total space $\mathcal E$,
called the \emph{horizontal distribution}, given by $\mathrm{Hor}(P^\mathcal E)_e=\mathrm{Ker}(P^\mathcal E_e)$. A vector $v\in T_e\mathcal E$ will be called \emph{horizontal} if it belongs to $\Hor(P^\mathcal E)_e$.

\subsection{Connections and exponential maps on Banach manifolds}
By a \emph{manifold with connection}, we mean a Banach
manifold $\mathcal M$ with a connection on its tangent bundle $\pi\colon T\mathcal M\to\mathcal M$.
If $P$ is a connection on $T\mathcal M$, one has a vector field $X(P)$ on $T\mathcal M$, called
\emph{geodesic field}, defined by the following: for $x\in\mathcal M$ and $v\in T_x\mathcal M$,
$X(P)_v$ is the unique horizontal vector in $T_v(T\mathcal M)$  that projects onto $v\in T_x\mathcal M$
(by the differential $\mathrm d\pi_v$). A curve $\gamma\colon I\to\mathcal M$ is a \emph{$P$-geodesic} if it is the projection of an integral curve $\Gamma\colon I\to T\mathcal M$ of $X(P)$.
If $\mathcal M$ is a manifold with connection $P$, then one has an exponential map
$$\exp^P\colon \mathrm{Dom}(\exp^P)\subset T\mathcal M\longrightarrow\mathcal M,$$
defined on an open subset $\mathrm{Dom}(\exp^P)\subset T\mathcal M$ containing the zero section, with
properties totally analogous to the exponential map of a connection on a finite-dimensional manifold.
In particular, for all $x\in\mathcal M$, $\mathcal A_x=\mathrm{Dom}(\exp^P)\cap T_x\mathcal M$ is a star-shaped open neighborhood of $0$ in $T_x\mathcal M$, and, denoting by $\exp^P_x$ the restriction of $\exp^P$ to $\mathcal A_x$, one has $\mathrm d\exp^P_x(0)=\mathrm{Id}$. In particular, $\exp^P_x$ is a diffeomorphism from an open neighborhood of $0$ in $T_x\mathcal M$ onto an open neighborhood of $x$ in $\mathcal M$.

\subsection{Banach bundles of sections of a finite-dimensional vector bundle}
\label{sub:naturalconnection}
We now describe an important example of the abstract setting of Subsection~\ref{sub:connectionsBanachbundles}.
Consider a vector bundle $\pi^E\colon E\to M$ over a finite-dimensional differentiable manifold $M$, and let $P^E\colon TE\to E$ be a connection in $E$. Let $D$ be a compact differentiable manifold (possibly with boundary), and for some $k\ge2$, set $\mathcal M=\mathcal C^k(D,M)$ and $\mathcal E=\mathcal C^k(D,E)$. There exists a natural map $\pi^\mathcal E\colon \mathcal E\to\mathcal M$, namely the map $(\pi^E)_*$ of left-composition with $\pi^E$. The sets $\mathcal M$ and $\mathcal E$ admit a natural structure of Banach manifold, making $\pi^\mathcal E\colon\mathcal E\to\mathcal M$ an infinite-dimensional Banach vector bundle. More precisely, for $f\in\mathcal M$, the fiber $\mathcal E_f$ is the Banach space of $\mathcal C^k$ sections of the pull-back bundle $f^*(E)$, also called sections of $E$ \emph{along $f$},
$$\mathcal E_f=\big\{F\in\mathcal C^k(D,E):F(x)\in E_{f(x)} \mbox{ for all } x\in D\big\}.$$
There is also a natural identification of the tangent bundle of $\mathcal E$ as $$T\mathcal E\cong\mathcal C^k(D,TE),$$ and if $f$ is in $\mathcal M$ (which can be thought of as the zero section of $\mathcal E$), there is a canonical splitting $T_f\mathcal E\cong T_f\mathcal M\oplus\mathcal E_f$ in horizontal and vertical parts respectively. The horizontal subspace of $T_f\mathcal E$ in this particular case is $T_f\mathcal M\cong\mathcal C^k(D,TM)$. To have a notion of horizontal subspace at the tangent space to $\mathcal E$ at points outside the zero section, we need a connection on this Banach vector bundle.

A connection $P^\mathcal E\colon T\mathcal E\to\mathcal E$ can be defined as being the map $(P^E)_*$ of left-composition with $P^E$. Let us show that this satisfies the axioms (a) and (b) described in Subsection~\ref{sub:connectionsBanachbundles}.

First, given $f\in\mathcal M$ and $F\in\mathcal E_f$, we have that the restriction $P_F^\mathcal E$ of $P^\mathcal E$ to $T_F\mathcal E$ maps a section $\eta$ of $F^*(TE)$ (i.e., an element of $T_F\mathcal E$) to the section $P^E\circ\eta$ of $f^*E$ (i.e., an element of $\mathcal E_f$), see the commutative diagram below. This map $P_F^\mathcal E$ is clearly linear, since it is given by left-composition with $P^E$, proving that axiom (a) holds.
\begin{equation*}
\xymatrix{& TE\ar[d]^{P^E} \\ & E\ar[d]^{\pi^E} \\ D\ar[r]^f\ar[ru]^F\ar@(u,l)[ruu]^\eta  & M}
\end{equation*}

Second, we observe that an atlas of trivializations of $\pi^\mathcal E\colon \mathcal E\to\mathcal M$ can be constructed
using smooth maps $s\colon\mathrm{Dom}(s)\subset D\times M\to\mathrm{Fr}(E)$ such that $\pi^E\circ s(p,x)=x$ for all $(p,x)\in\mathrm{Dom}(s)$,
and such that $s(p,\cdot)$ is a local trivialization of $\pi^E\colon E\to M$.
Once such a map $s$ is given, a trivialization $\mathfrak s\colon \mathrm{Dom}(\mathfrak s)\subset\mathcal C^k(D,M)\to\mathrm{Fr}_{\mathcal E_0}\big(\mathcal C^k(D,E)\big)$,
with $\mathcal E_0=\mathcal C^k(D,\mathds R^n)$,
is defined by setting, for all $x\in\mathrm{Dom}(\mathfrak s)=\big\{x\in\mathcal C^k(D,M):\mathrm{Gr}(x)\subset\mathrm{Dom}(s)\big\}$,\footnote{$\mathrm{Gr}(x)$ denotes the graph of $x\in\mathcal C^k(D,E)$.}
\begin{eqnarray*}
&\mathfrak s(x)\colon \mathcal C^k(D,\mathds R^n)\longrightarrow\mathcal C^k(D,E), & \\
&\mathfrak s(x)(v)_p=s\big(p,x(p)\big)v(p),&
\end{eqnarray*}
for all $v\in\mathcal C^k(D,\mathds R^n)$ and $p\in D$. Given $p\in D$, $F\in\mathcal E$ and $\eta\in T_e\mathcal E$, then:
\begin{multline*}
P^\mathcal E_F(\eta)(p)=P^E_{F(p)}\big(\eta(p)\big)\\=s\big(p,x(p)\big)\big[\mathrm d\widetilde s_p\big(F(p)\big)\eta(p)\big]+
\Gamma_{x(p)}^{P^E}\big(\mathrm d\pi^E_{F(p)}\big(\eta(p),F(p)\big)\big),
\end{multline*}
which says that the Christoffel symbol $\widetilde\Gamma$ of $P^\mathcal E$ associated to the trivialization
$\mathfrak s$ is given by:
\begin{eqnarray}\label{eq:tildeGamma}
&\widetilde\Gamma_x\colon\mathcal C^k(D,TM;x)\times\mathcal C^k(D,E;x)\longrightarrow\mathcal C^k(D,E;x)& \nonumber \\
&\widetilde\Gamma_x(v,e)(p)=\Gamma^P_{x(p)}\big(v(p),e(p)\big).&
\end{eqnarray}
Here, $\mathcal C^k(D,TM;x)$ and $\mathcal C^k(D,E;x)$ respectively denote the spaces of $\mathcal C^k$ sections of the pull-back bundles $x^*(TM)$ and $x^*(E)$.

An interesting particular case of the above construction is when $E=TM$ is the tangent bundle of $M$.
Recall that $D$ is a smooth manifold (possibly with boundary), $M$ is a manifold whose tangent bundle $TM$ has a
connection $P^{TM}$ and the Banach vector bundle $\mathcal E=\mathcal C^k(D,TM)$ is the tangent bundle of the Banach manifold $\mathcal M=\mathcal C^k(D,M)$, under the identification $$\mathcal E=\mathcal C^k(D,TM)\cong T\mathcal C^k(D,M)=T\mathcal M.$$ Endowing $T\mathcal M$ with the naturally induced connection $P^{T\mathcal M}$ described above, the $P^{T\mathcal M}$-geodesics in $\mathcal M$ are smooth curves $s\mapsto x_s\in\mathcal C^k(D,M)$ such that, for all $p\in D$, the curve $s\mapsto x_s(p)\in M$ is a $P^{TM}$-geodesic in $M$. This is a manifestation of the fact we will see next that the induced connection $P^{T\mathcal M}$ is characterized by every evaluation map $\ev_p\colon \mathcal M\to M$ being \emph{affine}, see Proposition~\ref{thm:univproperty}.

\subsection{Affine maps}
Let us now consider two Banach vector bundles $\pi^\mathcal E\colon \mathcal E\to\mathcal M$ and $\pi^{\mathcal E'}\colon \mathcal E'\to\mathcal M'$
endowed with connections $P^\mathcal E$ and $P^{\mathcal E'}$ respectively. Let $f\colon \mathcal M\to\mathcal M'$ be a smooth map and $T\colon \mathcal E\to\mathcal E'$ a smooth Banach bundle morphism for which the following diagram commutes.
\[\xymatrix{\mathcal E\ar[r]^T\ar[d]_{\pi^\mathcal E}&\mathcal E'\ar[d]^{\pi^{\mathcal E'}}\\ \mathcal M\ar[r]_f&\mathcal M'.}\]
\begin{defin}\label{thm:defaffine}
$T$ is said to be \emph{affine} if the following diagram commutes
\[\xymatrix{T\mathcal E\ar[r]^{\mathrm dT}\ar[d]_{P^\mathcal E}&T\mathcal E'\ar[d]^{P^{\mathcal E'}}\\\mathcal E\ar[r]_T&\mathcal E'.}\]
It is easy to see that $T$ is affine if and only if $\mathrm dT$ maps horizontal spaces to horizontal spaces.
\end{defin}

\begin{defin}\label{def:affine}
If $\mathcal M$ and $\mathcal M'$ are Banach manifolds endowed with connections $P$ and $P'$, a smooth map $f\colon \mathcal M\to\mathcal M'$ is \emph{affine} if $\mathrm df\colon T\mathcal M\to T\mathcal M'$ is affine.
\end{defin}

\begin{example}
\label{exa:naturalconnection}
Consider a finite-dimensional vector bundle $\pi^E\colon E\to M$ endowed with a connection $P^E$, let $D$ be a smooth manifold (possibly with
boundary) and consider the connection $P^\mathcal E$ defined on $\mathcal E=\mathcal C^k(D,E)$, as in Subsection~\ref{sub:naturalconnection}.
For all $p\in D$ denote by $\mathrm{ev}_p$ the evaluation at $p$ maps $\mathcal C^k(D,E)\to E$ and $\mathcal C^k(D,M)\to M$. Clearly, the following diagram commutes
\[\xymatrix{\mathcal E=\mathcal C^k(D,E)\ar[rr]^{\pi^\mathcal E}\ar[d]_{\mathrm{ev}_p}&&\mathcal C^k(D,M)=\mathcal M\ar[d]^{\mathrm{ev}_p}\\ E\ar[rr]_{\pi^E}&&M},\]
and it is easy to check that $\mathrm{ev}_p$ is affine.
Conversely, we now prove that this property characterizes the natural connection on $\mathcal E$ constructed in Subsection~\ref{sub:naturalconnection}.

\begin{prop}[Universal property of the natural connection]
\label{thm:univproperty}
The natural connection defined on $\pi^\mathcal E\colon \mathcal E\to\mathcal M$ as above is the unique connection for which $\mathrm{ev}_p$ is an affine map, for all $p\in D$.
\end{prop}

\begin{proof}
It follows readily from \eqref{eq:tildeGamma}. The condition that $\mathrm{ev}_p$ is affine is the commutativity of the following diagram:
\[\xymatrix{\mathcal C^k(D,TE)\ar[rr]^-{\mathrm d(\mathrm{ev}_p)}\ar[d]_{P^\mathcal E}&&TE\ar[d]^{P^E}\\\mathcal C^k(D,E)\ar[rr]_-{\mathrm{ev}_p}&&E.}\]
\end{proof}
\end{example}

\subsection{Invariance}
We conclude this section with a few results on affine maps.

\begin{prop}
\label{thm:invdiffeos}
Let $\pi^E\colon E\to M$ be a vector bundle with a connection $P^E$, let $D$ and $D'$ be manifolds (possibly with boundary), and set $\mathcal E=\mathcal C^k(D,E)$, $\mathcal E'=\mathcal C^k(D',E)$,  and $\mathcal M=\mathcal C^k(D,M)$. Let $\pi^\mathcal E\colon \mathcal E\to\mathcal M$ and $\pi^{\mathcal E'}\colon \mathcal E'\to\mathcal M$ be endowed with the associated connections $P^\mathcal E$ and $P^{\mathcal E'}$. If $\phi\colon D\to D'$ is a diffeomorphism of class $\mathcal C^k$, then the map
\[\phi^*\colon \mathcal E\longrightarrow\mathcal E'\]
of right-composition with $\phi$ is affine.
\end{prop}

\begin{proof}
It follows from the universal property of the natural connection, Proposition~\ref{thm:univproperty} (or directly from the definition).
\end{proof}

\begin{prop}\label{estens}
Let $\pi^E\colon E\to M$ and $\pi^{E'}\colon E'\to M'$ be vector bundles endowed with connections $P^E$ and $P^{E'}$ respectively, $D$ a smooth manifold (possibly with boundary). Set $\mathcal M=\mathcal C^k(D,M)$, $\mathcal M'=\mathcal C^k(D',M')$ and let $\mathcal E=\mathcal C^k(D,E)$, $\mathcal E'=\mathcal C^k(D,E')$ be endowed with the natural connections.

\begin{equation*}
\mbox{If }\;\vcenter{\vbox{\xymatrix{E\ar[r]^T\ar[d]&E'\ar[d]\\ M\ar[r]_f&M'} }}\;\mbox{ is affine, then } \;\vcenter{\vbox{\xymatrix{\mathcal E\ar[r]^{T_*}\ar[d]&\mathcal E'\ar[d]\\ \mathcal M\ar[r]_{f_*}&\mathcal M'}}} \;\mbox{ is affine.}
\end{equation*}
\end{prop}

\begin{proof}
Since the first diagram is affine, then the following diagram commutes
\[\xymatrix{TE\ar[r]^{\mathrm dT}\ar[d]_{P^E}&TE'\ar[d]^{P^{E'}}\\E\ar[r]_T&E'}\]
Taking left-composition with the above maps, and observing that $(\mathrm dT)_*=\mathrm d(T_*)$, we get the following commutative diagram, which proves the desired result.
\[\xymatrix{T\mathcal E\cong\mathcal C^k(D,TE)\ar[rr]^{(\mathrm dT)_*}\ar[d]_{(P^E)_*}&&\mathcal C^k(D,TE')\cong T\mathcal E'\ar[d]^{(P^{E'})_*}\\
\mathcal C^k(D,E)\ar[rr]_{T_*}&&\mathcal C^k(D,E')}\]
\end{proof}

\begin{cor}\label{thm:corrightcomp}
If $M$, $M'$ are manifolds with connections and $f\colon M\to M'$ is affine, then the map of left-composition
$f_*\colon\mathcal C^k(D,M)\to\mathcal C^k(D,M')$ is affine.
\end{cor}
\end{section}

\begin{section}{Slices for continuous affine actions}
\label{sec:slices}

In this section, we construct \emph{slices} for continuous affine actions of a Lie group on a Banach manifold. Let us consider the following setup:
\begin{itemize}
\item[(a)] $\mathcal M$ is a smooth Banach manifold,
\item[(b)] $G$ is a Lie group acting continuously by diffeomorphisms on $\mathcal M$,
\item[(c)] $x\in\mathcal M$ is a point where the action of $G$ is differentiable.
\end{itemize}
When $\mathcal M$ is endowed with a connection, we will say that the group action is \emph{affine}
if $G$ acts by affine diffeomorphisms of $\mathcal M$.

Define the auxiliary maps, with $g\in G$, $y\in\mathcal M$,
\begin{equation}\label{eq:defbetax}
\begin{aligned}
\beta_x\colon G &\longrightarrow\mathcal M\quad&\phi_g\colon\mathcal M &\longrightarrow\mathcal M \\
g&\longmapsto g\cdot x \quad& y&\longmapsto g\cdot y.
\end{aligned}
\end{equation}
From assumption (b), $\phi_g$ is a diffeomorphism for each $g\in G$. Assumption (c) means that $\beta_x$ is differentiable. In particular, the $G$-orbit of $x$ is a submanifold of $\mathcal M$, whose tangent space at $x$ is given by the image of $\mathrm d\beta_x(1)\colon \mathfrak g\to T_x\mathcal M$, where $\mathfrak g$ is the Lie algebra of $G$. Let us denote by $G_x$ the isotropy (or stabilizer) of $x$, which is the closed subgroup of $G$ given by $G_x=\big\{g\in G:\phi_g(x)=x\big\}$.

\begin{defin}\label{thm:defslice}
A \emph{slice} for the action of $G$ on $\mathcal M$ at $x$ is a smooth submanifold $\mathcal S\subset\mathcal M$ containing $x$, such that
\begin{enumerate}
\item\label{itm:proprslices1} the tangent space $T_x\mathcal S\subset T_x\mathcal M$ is a closed complement to $\mathrm{Im}\big(\mathrm d\beta_x(1)\big)$, i.e.,
$T_x\mathcal M=\mathrm{Im}\big(\mathrm d\beta_x(1)\big)\oplus T_x\mathcal S$;
\item\label{itm:proprslices2} $G\cdot\mathcal S$ is a neighborhood of the orbit $G\cdot x$, i.e., the orbit of every $y\in\mathcal M$ sufficiently close to $x$ must intersect $\mathcal S$;
\item\label{itm:proprslices3} $\mathcal S$ is invariant under the isotropy group $G_x$.
\end{enumerate}
\end{defin}

We will prove the existence of slices for affine actions of compact Lie groups. Towards this goal, we need an auxiliary result on linear actions of compact groups.

\begin{lem}\label{thm:preparatory}
Let $G$ be a compact Hausdorff topological group with a \emph{strongly continuous}\footnote{i.e., the maps $G\ni g\mapsto g\cdot x\in\mathcal X$ are continuous for all $x\in\mathcal X$.} linear representation
on a Banach space $\mathcal X$. Then:
\begin{itemize}
\item[(a)] if $S\subset\mathcal X$ is a closed $G$-invariant complemented subspace of $\mathcal X$, then
$S$ admits a $G$-invariant closed complement;\smallskip
\item[(b)] the origin of $\mathcal X$ has a fundamental system of $G$-invariant neighborhoods.
\end{itemize}
\end{lem}

\begin{proof}
First, observe that by the uniform boundedness principle, the linear operators on $\mathcal X$ associated to the action of elements $g\in G$ have norm bounded by a constant which is independent of $g$. By a simple argument, it follows that the action defines a continuous function $G\times\mathcal X\to\mathcal X$.
For part (a), let $P\colon\mathcal X\to\mathcal X$ be a projector (i.e., bounded linear idempotent) with image $S$. Define $\widetilde P\colon\mathcal X\to\mathcal X$ as the Bochner integral $\widetilde P(x)=\int_GgPg^{-1}x\,\mathrm dg$, where $\dd g$ is the Haar measure of $G$. It is easy to see that $\widetilde P$ is a well defined bounded linear operator on $\mathcal X$, with image contained in $S$. Furthermore, it
fixes the elements of $S$, and commutes with the $G$-action. It follows that $\widetilde P$ is also a projector with image $S$, and its kernel is the desired $G$-invariant closed complement to $S$.

As to part (b), let $V$ be an arbitrary neighborhood of the origin of $\mathcal X$. The inverse image of $V$ by the action $G\times\mathcal X\to\mathcal X$ is an open subset $Z$ of the product $G\times\mathcal X$ that contains $G\times\{0\}$. Since $G$ is compact, there exists a neighborhood $U$ of $0$ in $\mathcal X$ such that $G\times U$ is contained in $Z$, i.e., $g\cdot x\in V$ for all $g\in G$, $x\in U$. The union $\bigcup_{g\in G}gU$ is a $G$-invariant open neighborhood of $0$ in $\mathcal X$, contained in $V$.
\end{proof}

We also observe the following interesting fact.
\begin{lem}\label{thm:restrlocalaction}
Let $\rho\colon G\times \mathcal X\to\mathcal X$ be a continuous action of a compact group $G$ on a topological space $\mathcal X$. Assume $x_0\in\mathcal X$ is a fixed point of $G$. Then $x_0$ admits a fundamental system of $G$-invariant (open) neighborhoods.
\end{lem}

\begin{proof}
Let $V$ be an arbitrary neighborhood of $x_0$, and set $W=\bigcap_{g\in G}gV$. Clearly $W$ is $G$-invariant, and $W\subset V$. Let us show that $W$ is a neighborhood of $x_0$. The set $\rho^{-1}(V)$ is an open subset of $G\times\mathcal X$ that contains $G\times\{x_0\}$, and by the compactness of $G$, it also contains the product $G\times U$, where $U$ is some open neighborhood of $x_0$. Thus, $U\subset W$, and $W$ is a neighborhood of $x_0$. The interior of $W$ is also $G$-invariant.
\end{proof}

We can now prove our result on the existence of slices.
\begin{teo}\label{thm:existenceofslicecompact}
In the above situation, assume that $\mathcal M$ is endowed with a connection which is $G$-invariant (i.e., each diffeomorhism $\phi_g$ is affine), and that $G_x$ is compact. Then there exists a slice $\mathcal S$ through $x$.
\end{teo}

\begin{proof}
Consider the isotropy representation  of $G_x$ on $T_x\mathcal M$, given by $g\mapsto\mathrm d\phi_g(x)$. The finite-dimensional subspace $\mathrm{Im}\big(\mathrm d\beta_x(1)\big)$ is clearly invariant under this linear action. By part (a) of Lemma~\ref{thm:preparatory}, there exists a closed $G_x$-invariant complement $S$ of $\mathrm{Im}\big(\mathrm d\beta_x(1)\big)$. Denote by $\exp_x$ the exponential map of the $G$-invariant connection at $x$, and let $U_0\subset T_x\mathcal M$ be an open neighborhood of $0$ on which $\exp_x$ is a diffeomorphism. By part (b) of Lemma~\ref{thm:preparatory}, there exists an open neighborhood $\widetilde U_0\subset U_0$ of $0$ such that $\widetilde U_0\cap S$ is $G_x$-invariant. Set $$\mathcal S:=\exp_x(\widetilde U_0\cap S).$$

We claim that $\mathcal S$ is a slice for the action of $G$ at $x$. Property \eqref{itm:proprslices1} of slices is clearly satisfied, since $\mathrm d\exp_x(0)=\mathrm{Id}$, and $S$ is a closed
complement to $\mathrm{Im}\big(\mathrm d\beta_x(1)\big)$. Property \eqref{itm:proprslices2} would follow immediately from the transversality condition \eqref{itm:proprslices1} under the hypothesis of differentiability of the group action, which we do not assume. A slightly more involved topological argument based on degree theory is required for the continuous case, and this is discussed separately in Proposition~\ref{thm:nonemptyintersection}, which is to be applied with $A=N=\mathcal M$, $M=G$, $Q=\mathcal S$, $\chi$ being the action, $a_0=x$, and $m_0=1$. For property \eqref{itm:proprslices3}, observe that since the connection is $G$-invariant, then $\phi_g\circ\exp_x=\exp_{\phi_g(x)}\circ\,\mathrm d\phi_g(x)$, for all $g\in G$. Thus, given $v\in\widetilde U_0\cap S$ and $g\in G_x$, $\phi_g\big(\exp_x(v)\big)=\exp_x\big(\mathrm d\phi_g(x)v\big)\in\mathcal S,$ because $\widetilde U_0\cap S$ is $G_x$-invariant, i.e., $\mathcal S$ is $G_x$-invariant.
\end{proof}

\begin{prop}\label{thm:nonemptyintersection}
Let $M$ be a finite-dimensional manifold, $N$ a (possibly infinite-dimensional) Banach manifold,
$Q\subset N$ a Banach submanifold, and $A$ a topological space. Assume that $\chi\colon A\times M\to N$ is a continuous
function such that there exists $a_0\in A$ and $m_0\in M$ with:
\begin{itemize}
\item[(a)] $\chi(a_0,m_0)\in Q$;
\item[(b)] $\chi(a_0,\cdot)\colon M\to N$ is of class $\mathcal C^1$;
\item[(c)] $\partial_2\chi(a_0,m_0)\big(T_{m_0}M\big)+T_{\chi(a_0,m_0)}Q=T_{\chi(a_0,m_0)}N$.
\end{itemize}
Then, for $a\in A$ near $a_0$, $\chi(a,M)\cap Q\ne\emptyset$.
\end{prop}
\begin{proof}
Let $f\colon U\subset\mathds R^d\to\mathds R^d$ be a $\mathcal C^1$ function, defined on an open
neighborhood $U$ of $0$, such that $f(0)=0$ and $\mathrm df(0)$ an isomorphism. The induced
map $\widetilde f\colon \mathds S^{d-1}\to\mathds S^{d-1}$ is defined by $\widetilde f(x)=\Vert f(rx)\Vert^{-1}f(rx)$, where $r>0$ is such that $0$ is the unique zero of $f$ in the closed ball $\overline{B(0,r)}$ of $\mathds R^d$. This induced map must have topological degree equal to $\pm1$.

If $A$ is any topological space, $f\colon A\times U\to\mathds R^d$ is continuous, and $a_0\in A$ is
such that $f(a_0,\cdot)$ is of class $\mathcal C^1$, $f(a_0,0)=0$ and $\partial_2f(a_0,0)$ is an isomorphism, for $a$ near $a_0$, and $r>0$ sufficiently small, $0\in f\big(a,\overline{B(0,r)}\big)$. This follows from the continuity of the topological degree. The same conclusion holds for a function $f\colon A\times U\to\mathds R^d$, where now $U$ is an open neighborhood of $0$ in $\mathds R^s$, with $s\ge d$, under the assumption that $f(a_0,\cdot)$ is of class $\mathcal C^1$, $f(a_0,0)=0$, and $\partial_2f(a_0,0)$ be surjective. Namely, it suffices to apply the argument above to the function obtained by restricting $f$ to a $d$-dimensional subspace where $\partial_2f(a_0,0)$ is an isomorphism.

To finish the proof, use local coordinates adapted to $Q$ in $N$, and assume that $M$, $Q$ and $N$
are Banach spaces, with $N=Q\oplus\mathds R^d$,  $d\le s=\mathrm{dim}(M)$ is the codimension of $Q$, and $m_0=0$. In this situation, the conclusion is obtained applying the argument above to the function
$f\colon A\times M\to\mathds R^d$ given by $f(a,m)=\pi\big(\chi(a,m)\big)$, where $\pi\colon N\to\mathds R^d$ is the
projection relative to the decomposition $N=Q\oplus\mathds R^d$. Clearly, $f(a,m)=0$ if and only if
$\chi(a,m)\in Q$. Assumption (a) gives $f(a_0,0)=0$, and assumption (c) implies that $\partial_2f(a_0,0)$ is surjective.
\end{proof}

\subsection{Local actions}
\label{sub:localactions}
The existence of slices proved in Theorem~\ref{thm:existenceofslicecompact} holds in the more general case of \emph{local} group actions. Let us briefly recall the definition and a few basic facts about local actions.

Let $G$ be a Lie group and $\mathcal M$ a topological manifold. By a \emph{local action} of $G$ on $\mathcal M$, we mean a continuous map $\rho\colon \mathrm{Dom}(\rho)\subset G\times\mathcal M\to\mathcal M$, defined on an open subset $\mathrm{Dom}(\rho)\subset G\times\mathcal M$ containing $\{1\}\times\mathcal M$ satisfying:
\begin{itemize}
\item[(a)] $\rho(1,x)=x$ for all $x\in\mathcal M$;
\item[(b)] if $(g_2,x)\in\mathrm{Dom}(\rho)$ and $\big(g_1,\rho(g_2,x)\big)\in\mathrm{Dom}(\rho)$, then
$(g_1g_2,x)\in\mathrm{Dom}(\rho)$, and
$\rho\big(g_1,\rho(g_2,x)\big)=\rho(g_1g_2,x)$.
\end{itemize}
Usual group actions can be obtained as the particular case in which the domain $\mathrm{Dom}(\rho)$ coincides with the entire $G\times\mathcal M$. Local actions can be \emph{restricted}, in the following sense. If $\mathcal N\subset\mathcal M$ is a submanifold, then one has a local action $\widetilde\rho$ of $G$ on $\mathcal N$ by setting $\mathrm{Dom}(\widetilde\rho)=\big\{(g,x)\in(G\times\mathcal N)\cap\mathrm{Dom}(\rho):\rho(g,x)\in\mathcal N\big\}$, and $\widetilde\rho=\rho\vert_{\mathrm{Dom}(\widetilde\rho)}$. In fact, the most natural occurrence\footnote{In fact, local actions of groups are always restrictions of global actions. In the literature, these are known as \emph{enveloping actions} of the local action, see~\cite{Aba}.} of local actions is when one has a (global) action of a group $G$ on a topological manifold $\mathcal X$, and $\mathcal M$ is an open (not necessarily $G$-invariant)
subset of $\mathcal X$. The restriction of the action of $G$ to $\mathcal M$ in the above sense is a
local action of $G$ on $\mathcal M$.

Assumption (b) implies that for all $x\in M$, denoting by \[G_x=\big\{g\in G:(g,x)\in\mathrm{Dom}(\rho),\ \rho(g,x)=x\big\}\] the isotropy of $x$, then $G_x$ is a closed subgroup of $G$.

Given a local action $\rho$ of $G$ on $\mathcal M$, for $g\in G$, let $\rho_g$ denote the map $\rho(g,\cdot)$, defined on a (possibly empty) open set $\mathrm{Dom}(\rho_g)=\mathrm{Dom}(\rho)\cap\{g\}\times\mathcal M$. The following follow easily from the definition.

\begin{lem}\label{thm:lemlocalactions}
Let $\rho\colon \mathrm{Dom}(\rho)\subset G\times\mathcal M\to\mathcal M$ be a local action of $G$ on $M$. Then
\begin{itemize}
\item[(a)] for all $g\in G$, the map $\rho_g\colon \rho_g^{-1}\big(\mathrm{Dom}(\rho_{g^{-1}})\big)\to\rho_{g^{-1}}^{-1}\big(\mathrm{Dom}(\rho_g)\big)$ is a homeomorphism;
\item[(b)] the set $\big\{(g,x)\in G\times\mathcal M:x\in\rho_g^{-1}\big(\mathrm{Dom}(\rho_{g^{-1}})\big)\big\}$ is an open subset that contains $\{1\}\times\mathcal M$;
in particular:
\item[(c)] for all $x\in\mathcal M$, there exists an open neighborhood $U_x$ of $1$ in $G$ such that for all $g\in U_x$, $x\in\rho_g^{-1}\big(\mathrm{Dom}(\rho_{g^{-1}})\big)$.
\end{itemize}
\end{lem}
In view of (c) above, one can define a map $\beta_x\colon \mathrm{Dom}(\beta_x)\subset G\to\mathcal M$ on a neighborhood $\mathrm{Dom}(\beta_x)$ of $1$ in $G$, by $\beta_x(g)=\rho(g,x)$, compare with \eqref{eq:defbetax}. In particular, if $x\in\mathcal M$ is such that the map $\beta_x$ is differentiable (at $1$), then one has a well-defined linear map $\mathrm d\beta_x(1)\colon \mathfrak g\to T_x\mathcal M$. A subset $C\subset\mathcal M$ will be called $G$-invariant if, given $x\in C$, then $\rho(g,x)\in C$ for all $g\in\mathrm{Dom}(\beta_x)$.

In view of the above, the definition of slice for local actions is totally analogous to Definition~\ref{thm:defslice}. Furthermore, the statement and proof of Theorem~\ref{thm:existenceofslicecompact} carry over \emph{verbatim} to the case of local affine actions.
\end{section}

\begin{section}{Equivariant bifurcation}
\label{sec:mainbif}

Let us define what we intend by equivariant bifurcation. To simplify our discussion, we restrict to the case when there is a globally defined action (opposed to a local action). Consider the same setup (a), (b) and (c) of Section~\ref{sec:slices}. Let $[a,b]\ni\lambda\mapsto\mathfrak f_\lambda$ be a continuous path of $\mathcal C^k$-functionals $\mathfrak f_\lambda\colon\mathcal M\to\mathds R$, $k\ge2$, which are $G$-invariant, i.e., $\mathfrak f_\lambda(g\cdot y)=\mathfrak f_\lambda(y)$ for all $y\in\mathcal M$, $g\in G$ and $\lambda\in[a,b]$. We are interested in studying bifurcation of solutions to the equation $\dd\mathfrak f_\lambda(x)=0$.

\begin{defin}\label{def:eqbif}
Given $\lambda_0\in[a,b]$, we say that \emph{equivariant bifurcation} of critical points of the family $(\mathfrak f_\lambda)_{\lambda\in[a,b]}$ occurs at
$(x_{\lambda_0},\lambda_0)$ if there is a sequence $(x_n,\lambda_n)\in\mathcal M\times[a,b]$ such that
\begin{enumerate}
\item\label{itm:eqbif1} $\lim\limits_{n\to\infty}(x_n,\lambda_n)=(x_{\lambda_0},\lambda_0)$;\smallskip

\item\label{itm:eqbif2} $\mathrm d\mathfrak f_{\lambda_n}(x_n)=0$, for all $n$;\smallskip

\item\label{itm:eqbif3} $x_n\not\in G\cdot x_{\lambda_n}$, for all $n$.
\end{enumerate}
\end{defin}

We now discuss our central result, which is a sufficient condition for equivariant bifurcation in the above sense. It will be obtained by combining the slice theory developed in the previous section with a nonlinear formulation of a celebrated bifurcation result of J. Smoller and A. Wasserman~\cite{SmoWas}. In order to deal with the important general case of functionals defined on \emph{Banach} manifolds (rather than \emph{Hilbert} manifolds), we will need an appropriate framework described by a set of assumptions on an auxiliary Hilbert/Fredholm structure of the problem.

Let $B_2$ and $B_0$ be Banach spaces and $H$ be a Hilbert space with inner product $\langle\cdot,\cdot\rangle$. To keep things in perspective, in our geometric applications to a finite-dimensional manifold $M$, we will set $B_2=\mathcal C^{2,\alpha}(M)$, $B_0=\mathcal C^{0,\alpha}(M)$ and $H=L^{2}(M)$. Assume that $\mathcal M$ is modeled on $B_2$ and  is endowed with an affine $G$-invariant connection. Let $[a,b]\ni\lambda\mapsto x_\lambda\in\mathcal M$ be a continuous path, such that for all $\lambda$, $x_\lambda$ is a critical point of $\mathfrak f_\lambda$, which actually implies that the entire orbit $G\cdot x_\lambda$ consists of critical points of $\mathfrak f_\lambda$. Also, assume that a sufficiently small open set $U\subset\mathcal M$ containing all $x_\lambda$ admits continuous embeddings $U\subset B_0\subset H$, such that the following are satisfied. First, the local $G$-action on $U$ extends continuously to a local $G$-action on $B_0$ and on $H$. Second, there exists a continuous path $\lambda\mapsto\mathfrak d\mathfrak f_\lambda$ of $G$-equivariant $\mathcal C^{k-1}$-maps $\mathfrak d\mathfrak f_\lambda\colon U\to B_0$ satisfying
\begin{equation}\label{eq:defdstranof}
\mathrm d\mathfrak f_\lambda(y)\xi=\langle\mathfrak d\mathfrak f_\lambda(y),\xi\rangle,
\end{equation}
for all $y\in U$, $\xi\in T_yU\cong B_2$ and $\lambda$. In particular, we have
\[\mathfrak d\mathfrak f_\lambda(x_\lambda)=0,\quad\mbox{ for all }\lambda\in[a,b].\]
The map $\mathfrak d\mathfrak f_\lambda$ plays the role of the \emph{gradient} of $\mathfrak f_\lambda$, which does not exist in the usual sense due to the lack of a complete inner product on $B_2$.

For all $\lambda\in[a,b]$, let $G_\lambda$ be the isotropy of $x_\lambda$, which is a closed subgroup of $G$. Given $\varepsilon>0$, set
\begin{equation}\label{eq:nlambdae}
N_\lambda(\varepsilon):=\mathrm{span}\big\{v\in B_2:\mathrm d(\mathfrak d\mathfrak f_\lambda)_{x_\lambda}(v)=\mu v, \;\text{for some $\mu\le\varepsilon$}\big\}.
\end{equation}
We define the \emph{generalized negative eigenspace}\footnote{In the terminology of J. Smoller and A. Wasserman~\cite{SmoWas}, this is the \emph{neigenspace} of $\mathrm d(\mathfrak d\mathfrak f_\lambda)_{x_\lambda}$.} of $\mathrm d(\mathfrak d\mathfrak f_\lambda)_{x_\lambda}$ to be
\begin{equation}\label{eq:nlambda}
N_\lambda:=N_\lambda(0).
\end{equation}

Before stating the main result of this section, we briefly recall yet another notion used by J. Smoller and A. Wasserman~\cite[p. 73]{SmoWas}. A group $G$ is said to be \emph{nice} if, given unitary representations $\pi_1$ and $\pi_2$ of $G$ on Hilbert spaces $V_1$ and $V_2$ respectively, such that $B_1(V_1)/S_1(V_1)$ and $B_1(V_2)/S_1(V_2)$ have the same (equivariant) homotopy type as $G$-spaces, then $\pi_1$ and $\pi_2$ are equivalent. Here, $B_1$ and $S_1$ denote respectively the unit ball and the unit sphere in the specified Hilbert space, and the quotient $B_1(V_i)/S_1(V_i)$ is meant in the topological sense.\footnote{i.e., it denotes the unit ball of $V_i$ with its boundary contracted to one point.}

\begin{example}\label{ex:nice}
Any compact connected Lie group $G$ is nice. More generally, any compact Lie group with less than $5$ connected components is nice. Denoting by $G^0$ the identity connected component of $G$, then $G$ is nice if the discrete part $G/G^0$ is either the product of a finite number of copies of $\mathds Z_2$ (e.g., the case $G=\mathrm O(n)$); or the product of a finite number of copies of $\mathds Z_3$; or if $G/G^0=\mathds Z_4$, see~\cite{LeeWas}.
\end{example}

We are now ready to state and prove our main result.

\begin{teo}\label{thm:Gbifurcation}
In the above setup, assume that
\begin{itemize}
\item[(a)] there exists $\varepsilon>0$ such that $\mathrm{dim}\big(N_\lambda(\varepsilon)\big)<+\infty$, for all $\lambda\in[a,b]$;
\item[(b)] for all $\lambda$, $G_\lambda$ is a fixed compact nice subgroup $G_0$ of $G$;
\item[(c)] $\mathrm{Ker}\big(\mathrm d(\mathfrak d\mathfrak f_a)_{x_a}\big)=T_{x_a}(G\cdot x_a)$ and
$\mathrm{Ker}\big(\mathrm d(\mathfrak d\mathfrak f_b)_{x_b}\big)=T_{x_b}(G\cdot x_b)$;
\item[(d)] $\mathrm{dim}(N_a)\ne\mathrm{dim}(N_b)$.
\end{itemize}
Then, equivariant bifurcation of the family $(x_\lambda)_{\lambda}$ of critical points of $(\mathfrak f_\lambda)_{\lambda}$ occurs at some $(x_{\lambda_0},\lambda_0)$, with $\lambda_0\in\left]a,b\right[$.
\end{teo}

\begin{proof}
Under the above hypotheses, Theorem \ref{thm:existenceofslicecompact} ensures the existence of a slice $\mathcal S$ invariant under the action of $G_{0}$ by diffeomorphisms. We have a family $T_{\lambda}=\mathrm d \mathfrak f_\lambda $ of $G_{0}$-equivariant sections of $T\mathcal S$. Note that, since $\mathfrak f_\lambda$ is constant along the orbits and by the transversality property (1) of the slice, we have that $\mathcal S$ is a \emph{natural constraint}. In other words, the constrained critical points of the restriction $\mathfrak f_\lambda\big\vert_\mathcal S$ of $\mathfrak f_\lambda$ to $\mathcal S$ actually satisfy $\mathrm d \mathfrak f_\lambda (x_{\lambda})=0.$ Assumption (c) means that $x_{a}$ and $x_{b}$ are (equivariantly) nondegenerate critical points. The result then follows from \cite[Thm 2.1]{SmoWas}, in its nonlinear formulation explained in Appendix~\ref{app:A}.
\end{proof}

\begin{rem}\label{thm:fredhassumption}
Assumption (a) in Theorem~\ref{thm:Gbifurcation} is satisfied, for instance, when
$\lambda\mapsto\mathrm d\,\mathfrak d\mathfrak f_\lambda(x_\lambda)\colon B_2\to B_0$
is a continuous path of Fredholm operators that are \emph{essentially positive}. By definition, this means that $\mathrm d\mathfrak d\,\mathfrak f_\lambda(x_\lambda)$ are Fredholm operators of the form $P_\lambda+K_\lambda$, where $P_\lambda\colon B_2\to B_0$ is a symmetric isomorphism (relatively to the inner product of $H$) and satisfies $\langle P_\lambda x,x\rangle>0$ for all $x\in B_2\setminus\{0\}$; and $K_\lambda\colon B_2\to B_0$ is a compact symmetric operator (also relatively to to the inner product of $H$). In this situation, the space $N_\lambda(\varepsilon)$ is the direct
sum of the eigenspaces of the compact operator $P_\lambda^{-1}K_\lambda$ (which is symmetric with respect to the inner
product defined by $P_\lambda$, hence diagonalizable) corresponding to eigenvalues less than or equal to $\varepsilon-1<0$. Assuming $\varepsilon<1$, the operator $P_\lambda^{-1}K_\lambda$ has only a finite number of such eigenvalues, and each of them has finite multiplicity. By continuity, one can give a uniform estimate on the dimension of $N_\lambda(\varepsilon)$, for $\lambda\in[a,b]$.
\end{rem}

Assumption (d) in Theorem~\ref{thm:Gbifurcation} means that there is a \emph{jump} of the Morse index of $x_\lambda$, as $\lambda$ goes from $a$ to $b$. We now present a subtler criterion for equivariant bifurcation, where this assumption is weakened. Recall the \emph{isotropy representation} $\pi_\lambda$ of $G_\lambda$ on $T_{x_\lambda}\mathcal M$ is the linear representation defined by $\pi_\lambda(g)=\mathrm d\phi_g(x_\lambda)$. Since $\mathfrak d\mathfrak f_\lambda$ is equivariant, it is easy to see that $N_\lambda(\varepsilon)$ is invariant under $\pi_\lambda$, for all $\varepsilon>0$. Define the \emph{negative isotropy representation} $\pi_\lambda^-$ to be the restriction
\begin{equation}\label{eq:negisorep}
\pi_\lambda^-:=\pi_\lambda\vert_{N_\lambda}\colon N_\lambda\longrightarrow N_\lambda.
\end{equation}
Observe that $\dim N_\lambda$ is the Morse index of $x_\lambda$.

\begin{teo}\label{thm:Gbifurcation2}
Replace the assumption {\rm (d)} of Theorem~\ref{thm:Gbifurcation} with
\begin{itemize}
\item[(d')] the negative isotropy representations $\pi_a^-$ and $\pi_b^-$ are not equivalent.\footnote{Two representations $\pi_i:H\to\mathrm{GL}(V_i)$, $i=1,2$, of the group $H$ on the vector spaces $V_1$ and $V_2$ respectively, are \emph{equivalent} if there exists a $H$-equivariant isomorphism $T:V_1\to V_2$, i.e., an isomorphism satisfying $T\big(\pi_1(h)v\big)=\pi_2(h)\big(T(v)\big)$ for all $h\in H$ and all $v\in V_1$. In particular, $\dim V_1=\dim V_2$.}
\end{itemize}
Then, the same conclusion holds, i.e., equivariant bifurcation of $(x_\lambda)_{\lambda\in[a,b]}$ occurs at some $(x_{\lambda_0},\lambda_0)$, with $\lambda_0\in\left]a,b\right[$.
\end{teo}

\begin{proof}
The same proof of Theorem~\ref{thm:Gbifurcation} applies, using \cite[Thm 3.3]{SmoWas}, in its nonlinear formulation (explained in Appendix~\ref{app:A}), to obtain the conclusion.
\end{proof}
\end{section}

\begin{rem}
All the results stated above carry over \emph{verbatim} to the case of local affine actions, using the same standard procedures mentioned before.
\end{rem}

\begin{section}{Geometric applications on CMC hypersurfaces}
\label{sec:geomappl}

In this section, we apply our abstract equivariant bifurcation results (Theorems~\ref{thm:Gbifurcation} and \ref{thm:Gbifurcation2}) to the geometric variational problem of \emph{constant mean curvature (CMC)} embeddings. Bifurcation phenomena for $1$-parameter families of CMC embeddings have been studied in the last years by several authors, see, e.g., \cite{AliPic11,bp1,KoiPalPic2011,LimLimPic2011}. We will state and prove general bifurcation results for CMC embeddings (Theorems~\ref{thm:MorseCMCbif} and \ref{thm:equivCMCbifurcation}) and discuss how some explicit bifurcation examples can be reobtained from these general results.

\subsection{Variational setup}
The problem of finding constant mean curvature $H$ embeddings of a compact $m$-dimensional manifold $M$ into a complete Riemannian manifold $(\overline M,\overline g)$ with $\mathrm{dim}(\overline M)=m+1$ is equivalent to finding critical points of the area functional with a fixed volume constraint, where $H$ is the Lagrange multiplier (which will play the role of the parameter $\lambda$). More precisely, assume for simplicity that $M$ and $\overline M$ are oriented, and consider the $1$-parameter family of functionals $(\mathfrak f_H)_H$ given by
\begin{equation}\label{eq:fh}
\mathfrak f_H(x)=\Area(x)+mH\,\Vol(x),
\end{equation}
where $x\colon M\to\overline M$ is an embedding, $\Area(x)=\int_M \vol_{x^*(\overline g)}$ is the volume of $x(M)\subset\overline M$, $\vol_{x^*(\overline g)}$ is the volume form of the pull-back metric $x^*(\overline g)$ and $\Vol(x)$ is the \emph{volume enclosed}\footnote{This notion will be clarified by the end of this subsection. For now, one may assume for simplicity that $x(M)=\partial\Omega$ is the boundary of an open bounded region $\Omega\subset\overline M$, and then $\Vol(x)=\int_\Omega \vol_{\overline g}$ is the volume of this enclosed region.} by $x(M)$. Then $x\colon M\to\overline M$ is a critical point of $\mathfrak f_H$ if and only if it is an embedding of constant mean curvature $H$, see \cite{BardoC, BardoCEsc}. As we will see later, a convenient regularity assumption is that $\mathfrak f_H$ acts on the space of H\"older $\mathcal C^{2,\alpha}$ embeddings.

More precisely, assuming that the embedding $x$ is transversely oriented (i.e., the normal bundle to $x$ is oriented), we may parameterize embeddings close to $x$ by functions on $M$ using the normal exponential map. An embedding $x_f\colon M\to\overline M$ that is $\mathcal C^{2,\alpha}$-close to $x$ can be written as
\begin{equation}
x_f(p)=\exp^\perp_{x(p)}\big(f(p)\,N_x(p)\big), \quad p\in M,
\end{equation}
where $\exp^\perp$ is the normal exponential map of $x(M)\subset\overline M$ and $N_x$ is a unit normal vector field along $x$.
We thus identify $x_f$ with the function $f\in\mathcal C^{2,\alpha}(M)$, which is close to zero. This also gives an identification of the tangent space at $x$ to the space of $\mathcal C^{2,\alpha}$ embeddings (which is formed by normal vector fields along $x$) with the Banach space $\mathcal C^{2,\alpha}(M)$. With this identification, the first variation formula for \eqref{eq:fh} is given by
\begin{equation}\label{eq:firstvar}
\dd\mathfrak f_H(x)(f)=\int_M m\big(H-\mathcal H(x)\big) f\,\vol_{x^*(\overline g)}, \quad f\in\mathcal C^{2,\alpha}(M),
\end{equation}
where $\mathcal H(x)$ is the mean curvature function of the embedding $x$. From \eqref{eq:firstvar}, it follows that $x$ is a critical point of $\mathfrak f_H$ if and only if $\mathcal H(x)=H$ (i.e., $x$ has constant mean curvature $H$), as we claimed above.

We will also need to consider the second variation of \eqref{eq:fh} at a critical point $x$, which under the above identifications, is the symmetric bilinear form on $\mathcal C^{2,\alpha}(M)$ given by
\begin{equation}\label{eq:secondvar}
\dd^2\mathfrak f_H(x)(f_1,f_2)=-\int_M J_x(f_1)f_2 \,\vol_{x^*(\overline g)}, \quad f_1,f_2\in\mathcal C^{2,\alpha}(M),
\end{equation}
where $J_x$ is the Jacobi operator
\begin{equation}\label{eq:JacobioperatorCMC}
J_x=\Delta_x+\| A_x\|^2+m\,\mathrm{Ric}_{\overline M}(N_x),
\end{equation}
where $\Delta_x$ is the Laplacian of the pull-back metric $x^*(g)$ on $M$, $\|A_x\|$ is the norm of the second fundamental form of $x$, $\mathrm{Ric}_{\overline M}$ is the (normalized) Ricci curvature of the ambient space $(\overline M,\overline g)$ and $N_x$ is a unit normal field along $x$. Functions $f$ in the kernel of $J_x$ are called \emph{Jacobi fields} along $x$. The number of negative eigenvalues of $J_x$ (counted with multiplicity) is the \emph{Morse index} of $x$, that we denote $\mathrm i_{\mathrm{Morse}}(x)$.

The ambient isometry group $G=\Iso(\overline M,\overline g)$ acts on the space of embeddings, and composing a CMC embedding with an element of $G$ trivially gives rise to a new CMC embedding. Recall that from the Myers-Steenrod Theorem, $G$ is a Lie group, and is compact if $\overline M$ is compact (see \cite{Koba}). In addition, since $(\overline M,\overline g)$ is complete, the Lie algebra of $G$ is identified with the space of Killing vector fields of $(\overline M,\overline g)$. We are interested in \emph{$G$-equivariant} bifurcation of CMC embeddings, i.e., getting new embeddings that are not merely obtained by composing a pre-existing one with an isometry of the ambient manifold. Another way in which one could trivially obtain a new CMC embedding is by reparameterizing it, i.e., composing on the right with a diffeomorphism of $M$. Two CMC embeddings $x_i\colon M\to\overline M$, $i=1,2$, are said to be \emph{isometrically congruent} if there exists a diffeomorphism $\phi$ of $M$ and an isometry $\psi$ of $(\overline M,\overline g)$ such that $x_2=\psi\circ x_1\circ\phi$.

Infinitesimally, the action of $G$ provides some trivial Jacobi fields along any critical point. Namely, if $K$ is a Killing vector field of $(\overline M,\overline g)$, then $f=\overline g(K,N_x)$ is a Jacobi field along $x$. Denote by $\mathrm{Jac}_x$ the (finite-dimensional) vector space of Jacobi fields along $x$, and by $\mathrm{Jac}^K_x$ the subspace of $\mathrm{Jac}_x$ spanned by the functions $\overline g(K,N_x)$, where $K$ is a Killing vector field of $(\overline M,\overline g)$. The CMC embedding $x$ will be called \emph{nondegenerate} if $\mathrm{Jac}^K_x=\mathrm{Jac}_x$, i.e., if every Jacobi field along $x$ arises from a Killing field of the ambient space.

It is natural to expect that, with the above equivariant notion of nondegeneracy, an equivariant implicit function theorem should hold. Indeed, the following is proved in~\cite[Prop 4.1]{BPS1}.

\begin{teo}\label{thm:abstractIFT}
Let $x\colon M\to\overline M$ be a nondegenerate CMC embedding, with mean curvature equal to $H_0$. Then, there exists $\varepsilon>0$ and
a smooth map $$\left]H_0-\varepsilon,H_0+\varepsilon\right[\ni H\longmapsto x_H\in\mathcal C^{2,\alpha}(M,\overline M),$$ such that for all $H$, $x_H\colon M\to\overline M$ is a CMC embedding of mean curvature $H$ and
\begin{itemize}
\item[(a)] $x_{H_0}=x$;\smallskip
\item[(b)] if $y\colon M\to\overline M$ is a CMC embedding sufficiently close to $x$ in the $\mathcal C^{2,\alpha}$-topology, then there exists
$H\in\left]H_0-\varepsilon,H_0+\varepsilon\right[$ such that $y$ is isometrically congruent to $x_H$.
\end{itemize}
\end{teo}

\subsection{A few technicalities}\label{subsec:tech}
Let us now deal with some technicalities we omitted in the above explanation of the variational setup for the CMC problem.

First, we need to give a more precise definition of the space where \eqref{eq:fh} is defined. For reasons\footnote{This choice has to do with the nature of the second variation of $\mathfrak f_H$, which we will want to be a Fredholm operator under the appropriate identification.} that will later be clear, it is convenient to consider the space of embeddings $x\colon M\to\overline M$ endowed with a $\mathcal C^{2,\alpha}$-topology. More precisely, consider the set $\mathrm{Emb}^{2,\alpha}(M,\overline M)$ of embeddings of class $\mathcal C^{2,\alpha}$ of $M$ into $\overline M$. This is an open subset of the Banach manifold $\mathcal C^{2,\alpha}(M,\overline M)$, and hence inherits a natural differential structure, becoming a Banach manifold. Since we want to identify embeddings obtained by reparameterizations of a given embedding, we have to consider the action of the group $\Diff(M)$ of diffeomorphisms of $M$ by right-composition on $\mathrm{Emb}^{2,\alpha}(M,\overline M)$. We denote the orbit space of this action by
\begin{equation}\label{eq:emm}
\mathcal E(M,\overline M)=\mathrm{Emb}^{2,\alpha}(M,\overline M)/\Diff(M),
\end{equation}
and its elements are called \emph{unparameterized embeddings}. This set has the structure of an infinite-dimensional \emph{topological} manifold modeled on the Banach space $\mathcal C^{2,\alpha}(M)$. Its geometry and local differential structure are studied in detail in \cite{AliPic10}. Given $x\in\mathcal C^{2,\alpha}(M,\overline M)$, we denote by $[x]$ its class in $\mathcal E(M,\overline M)$. Henceforth, we are assuming for simplicity that $x(M)$ is transversely oriented in $\overline M$.

If we take $x\in\mathrm{Emb}^{2,\alpha}(M,\overline M)$ in the dense subset of \emph{smooth} embeddings, there exists an open neighborhood of $[x]$ in $\mathcal E(M,\overline M)$ and a bijection from this neighborhood to a neighborhood of the origin of $\mathcal C^{2,\alpha}(M)$, whose image is identified (using the normal exponential map) with $\mathcal C^{2,\alpha}$ embeddings equivalent to $x$ under the action of $\Diff(M)$. As $x$ runs in the set of smooth embeddings, those maps form an atlas for $\mathcal E(M,\overline M)$ whose charts are continuously compatible. Moreover, if a smooth functional defined in $\mathrm{Emb}^{2,\alpha}(M,\overline M)$ is invariant under $\Diff(M)$, then the induced functional in $\mathcal E(M,\overline M)$ is smooth\footnote{As a map from a neighborhood of the origin in $\mathcal C^{2,\alpha}(M)$ to $\mathds R$.} in every local chart. Using these charts, we also have an identification
\begin{equation}\label{eq:idtxe}
T_{[x]}\mathcal E(M,\overline M)\cong\mathcal C^{2,\alpha}(M)
\end{equation}
of this tangent space with the Banach space of (real-valued) $\mathcal C^{2,\alpha}$ functions on $M$. For more details on this standard construction, see \cite{AliPic10}.

Second, note that, if $x\colon M\to\overline M$ is an embedding, unless $x(M)\subset\overline M$ is the boundary of a bounded open set of $\overline M$, then the \emph{enclosed volume} $\Vol(x)$ is not well-defined. Moreover, it is not clear that such quantity should be invariant under the action of $G$. To overcome these problems, we recall the notion of \emph{invariant volume functionals} for embeddings of $M$ into $\overline M$ developed in \cite[Appendix B]{BPS1}.

\begin{defin}\label{thm:definvariantvolumes}
Let $\mathcal U\subset\mathcal C^{2,\alpha}(M,\overline M)$ be an open subset of embeddings $x\colon M\to\overline M$. An \emph{invariant volume functional} on $\mathcal U$
is a real valued function $\Vol\colon\mathcal U\to\mathds R$ satisfying:
\begin{itemize}
\item[(a)] $\Vol(x)=\int_Mx^*(\eta)$, where $\eta$ is an $m$-form defined on an open subset $U\subset\overline M$ such that
$\mathrm d\eta=\mathrm{vol}_{\overline g}$ is the volume form of $\overline g$ in $U$;
\smallskip

\item[(b)] given $x\in\mathcal U$, for all $\phi\in\Iso(\overline M,\overline g)$ sufficiently close to the identity, $\Vol(\phi\circ x)=\Vol(x)$.
\end{itemize}
If $M$ has boundary, the invariance property (b) is required to hold only for isometries $\phi$ near the identity that preserve
$x(\partial M)$, i.e., $\phi\big(x(\partial M)\big)=x(\partial M)$. An embedding will be called \emph{regular} if it is contained in some open set $\mathcal U$ of $\mathcal C^{2,\alpha}(M,\overline M)$ which is the domain of some invariant volume functional.
\end{defin}

\begin{example}
If $x(M)$ is the boundary of a bounded open subset of $\overline M$, then $x$ is regular.
If $\overline M$ is non compact, and $\mathrm{Iso}(\overline M,\overline g)$
is compact, then every  embedding into $\overline M$ is regular.
If $x\colon M\to\overline M$ has image contained in some open subset $U\subset\overline M$ whose $m$-th
de Rham cohomology vanishes, then $x$ is regular. In particular,
if $\overline M=\mathds R^{m+1}$ or $\overline M=\mathds S^{m+1}$, then every embedding into $\overline M$ is regular, see \cite[Ex 5]{BPS1}.
\end{example}

Third, when considering an invariant volume functional as above (defined in a neighborhood of a given embedding), the left-composition action of $\Iso(\overline M,\overline g)$ has to be restricted to this domain, giving rise to a local action. As remarked above, standard techniques apply to have the necessary results also in the case of local actions.

With the above considerations on the (topological) manifold $\mathcal E(M,\overline M)$ of unparameterized embeddings of class $\mathcal C^{2,\alpha}$ and the local existence of an invariant volume functional, we may study the CMC variational problem in this precise global analytical setup. The functional \eqref{eq:fh} is then well-defined and smooth in a neighborhood of a smooth unparameterized regular embedding, and formulas \eqref{eq:firstvar} and \eqref{eq:secondvar} hold with the appropriate identifications above mentioned.

\subsection{Equivariant bifurcation using Morse index}

We will now use our abstract equivariant bifurcation result to obtain a bifurcation result for CMC embeddings when there is a jump of the Morse index. Let us recall some basic terminology. Assume that $[a,b]\ni r\mapsto x_r\in\mathcal C^{2,\alpha}(M,\overline M)$ is a continuous family of CMC embeddings of $M$ into $\overline M$ (which already implies that $x_r\colon M\to\overline M$ is smooth\footnote{It is well-known that CMC hypersurfaces are the solution to a quasilinear elliptic PDE, hence smoothness follows from usual elliptic regularity theory.}) and let $H_r$ denote the value of the mean curvature of $x_r$. An instant $r_*\in\left]a,b\right[$ is a \emph{bifurcation instant} for the family $(x_r)_{r\in[a,b]}$ if there exists a sequence $r_n$ in $[a,b]$ tending to $r_*$ as $n\to\infty$ and a sequence $x_n$ of CMC embeddings of $M$ into $\overline M$, with the mean curvature of $x_n$ equal to $H_{r_n}$, such that $x_n$ tends to $x_{r_*}$ in $\mathcal C^{2,\alpha}(M,\overline M)$ as $n\to\infty$ and for every $n$, $x_n$ is \emph{not} isometrically congruent to $x_{r_n}$.

Given a CMC embedding $x\colon M\to\overline M$, let $G_x$ denote the closed subgroup of
$\mathrm{Iso}(\overline M,\overline g)$ consisting of isometries $\psi$ that leave $x(M)$ invariant, i.e., such that $\psi\big(x(M)\big)\subset x(M)$. In other words, $G_x$ is the \emph{isotropy} of $x$ under the action of $G$. Since $M$ is compact and the action of $G$ is proper, $G_x$ is compact. The Lie algebra $\mathfrak g_x$ of $G_x$ is identified with the space of Killing vector fields of $(\overline M,\overline g)$ that are everywhere tangent to $x(M)$. The codimension of $G_x$ in $G$ is equal to the dimension of $\mathrm{Jac}^K_x$.

\begin{teo}\label{thm:MorseCMCbif}
Let $[a,b]\ni r\mapsto x_r\in\mathcal C^{2,\alpha}(M,\overline M)$ be a $\mathcal C^1$-map, where $x_r\colon M\to\overline M$ is a regular CMC embedding
for all $r$, having mean curvature equal to $H_r$. Let $r_*\in\left]a,b\right[$ be an instant where
\begin{itemize}
\item[(a)] the derivative $H'_{r_*}$ of the map $[a,b]\ni r\mapsto H_r\in\mathds R$ at $r_*$ is nonzero;
\smallskip

\item[(b)] for $\varepsilon>0$ small enough:
\smallskip

\begin{itemize}
\item[(b1)] $x_{r_*-\varepsilon}$ and $x_{r_*+\varepsilon}$ are nondegenerate;
\smallskip

\item[(b2)] the identity connected component $G^0_r$ of the isotropy $G_{x_r}$ does not depend on
$r$, for $r\in\left[r_*-\varepsilon,r_*+\varepsilon\right]$;
\smallskip

\item[(b3)] $\mathrm i_\mathrm{Morse}(x_{r_*-\varepsilon})\ne\mathrm i_\mathrm{Morse}(x_{r_*+\varepsilon})$.
\end{itemize}
\end{itemize}
Then, $r_*$ is a bifurcation instant for the family $(x_r)_r$.
\end{teo}

\begin{proof}
We first verify that the CMC variational problem satisfies the required conditions and then use Theorem~\ref{thm:Gbifurcation} to obtain the conclusion. In the notation of Section~\ref{sec:mainbif}, we have $B_2=\mathcal C^{2,\alpha}(M)$, $B_0=\mathcal C^{0,\alpha}(M)$ and $H=L^2(M,\nu)$, where $\nu$ is an arbitrarily fixed volume form (or density) on $M$. It will be convenient to choose $\nu$ to be the volume form of the pull-back metric $x_{r_*}^*(\overline g)$.

Let $\overline\nabla$ be the Levi-Civita connection of $(\overline M,\overline g)$. Using this connection, one can define an associated natural connection on $\mathrm{Emb}^{2,\alpha}(M,\overline M)$, as in Example~\ref{exa:naturalconnection}. This connection is defined on the entire manifold $\mathcal C^{2,\alpha}(M,\overline M)$, and is characterized by the fact that the evaluation maps $\ev_p\colon\mathcal C^k(M,\overline M)\to\overline M$, $p\in M$, are affine (Proposition~\ref{thm:univproperty}).\footnote{An explicit description of the horizontal distribution of this connection is given as follows. The tangent bundle of $\mathcal C^{2,\alpha}(M,\overline M)$ can be naturally identified with $\mathcal C^{2,\alpha}(M,T\overline M)$; an element of $\mathcal C^{2,\alpha}(M,T\overline M)$ is a map of class $\mathcal C^{2,\alpha}$ from $M$ to $T\overline M$, which is a vector field of class $\mathcal C^{2,\alpha}$ in $\overline M$ along some function $f\colon M\to\overline M$ of class $\mathcal C^{2,\alpha}$. The tangent space to $\mathcal C^{2,\alpha}(M,T\overline M)$ at the point $X$ is the space of vector fields of class $\mathcal C^{2,\alpha}$ in $T\overline M$ along $X$, i.e., maps $\eta\colon M\to T(T\overline M)$ of class $\mathcal C^{2,\alpha}$ such that $\eta(p)$ is a tangent vector to $T\overline M$ at the point $X(p)$, for all $p\in M$. The vertical subspace is given by those $\eta$'s such that $\eta(p)$ is vertical, for all $x\in M$. The horizontal subspace is the space of maps $\eta$ such that $\eta(p)$ is horizontal relatively to the connection $\overline\nabla$ of $\overline M$ for all $p\in M$.}

Let $G$ be the identity connected component of $\Iso(\overline M,\overline g)$, which is a (finite-dimensional) Lie group,
and consider the smooth action of $G$ by left-composition on $\mathcal C^{2,\alpha}(M,\overline M)$.
Clearly, $\mathrm{Emb}^{2,\alpha}(M,\overline M)$ is invariant by left-compositions with diffeomorphisms
of $\overline M$, so we have an induced action of $G$ on $\mathrm{Emb}^{2,\alpha}(M,\overline M)$.
Since isometries preserve the Levi-Civita connection, the actions of $G$ on both $\mathcal C^{2,\alpha}(M,\overline M)$
and $\mathrm{Emb}^{2,\alpha}(M,\overline M)$ are by affine diffeomorphisms, see Proposition \ref{estens}.
We observe furthermore that the left-action of $G$ on $\mathrm{Emb}^{2,\alpha}(M,\overline M)$ commutes
with the right-action of the diffeomorphism group $\Diff(M)$. This implies that one can define a left-action
of $G$ on the quotient space $\mathcal E(M,\overline M)$. Finally, we recall from Proposition~\ref{thm:invdiffeos}
that the right-action of $\Diff(M)$ on $\mathrm{Emb}^{2,\alpha}(M,\overline M)$ is
by affine diffeomorphisms, so that one has an induced connection on $\mathcal E(M,\overline M)$
which is preserved by the action of $G$.

Let $x\colon M\to\overline M$ be a $\mathcal C^{2,\alpha}$ embedding. Since the action of $G$ on $\overline M$ is proper, and $M$ is compact, then the isotropy group $G_x$ is a compact subgroup of $G$. We recall from \cite{AliPic10} that there exists a natural (topological) atlas of  continuously compatible charts of $\mathcal E(M,\overline M)$ such that, in these charts, the (local) action of $G$ is differentiable at the class $[x]$ of every \emph{smooth} embedding $x\colon M\to\overline M$. In particular, if $x$ has constant mean curvature, then the action of $G$ on $\mathcal E(M,\overline M)$ is differentiable at $[x]$. Moreover, by Lemma~\ref{thm:restrlocalaction}, $[x]$ admits arbitrarily small neighborhoods in $\mathcal E(M,\overline M)$ that are invariant by $G_x$. With this, we are in the variational framework described in Axioms (a), (b) and (c) of Section~\ref{sec:slices}.

By assumption (a), there exists a $\mathcal C^1$ function $H\mapsto r(H)$, defined in a neighborhood of $H_{r_*}$, with the property that the (constant) mean curvature of $x_{r(H)}$ is equal to $H$, for all $H$ in this neighborhood. Thus, we may assume that the CMC embeddings $x_r$, for $r$ close to $r_*$, are parameterized by their mean curvature $H_r$ instead of $r$, and we write $x_{H_r}$. Consider an invariant volume functional $\Vol$ defined in a neighborhood $\mathcal U\subset\mathcal C^{2,\alpha}(M,\overline M)$ of $x_{H_{r_*}}$. For $H$ near $H_{r_*}$ consider the one-parameter family of functional $\mathfrak f_H\colon \mathcal U\to\mathds R$ given by \eqref{eq:fh}. The group $G$ acts by affine diffeomorphisms on the manifold $\mathcal C^{2,\alpha}(M,\overline M)$ by left-composition; in particular, we have a local action on the open subset $\mathcal U$. Since both $\Area$ and $\Vol$ are invariant under composition on the right with isometries
of $(\overline M,\overline g)$, then $\mathfrak f_H$ is invariant under the local action of $G$.
Moreover, $\Area$ and $\Vol$ are invariant under right-composition with diffeomorphisms of $M$, so $\mathfrak f_H$ gives a well-defined smooth functional on the quotient space $\mathcal E(M,\overline M)$, as discussed before in Subsection~\ref{subsec:tech}. With the appropriate identifications, the first variation formula for this functional is given by \eqref{eq:firstvar}, which means that the map $\mathfrak d\mathfrak f_H(x)\colon U\subset B_2\to B_0$ defined in \eqref{eq:defdstranof} is
\begin{equation}\label{eq:dstranoCMC}
\mathfrak d\mathfrak f_H(x)=m\big(H-\mathcal H(x)\big)\psi_x,
\end{equation}
where $\psi_x\colon M\to\mathds R^+$ is the unique map satisfying $\psi_x\,\mathrm {vol}_{(x_{H_{r_*}})^*(\overline g)}=\mathrm {vol}_{x^*(\overline g)},$ in particular, $\psi_{x_{H_{r_*}}}\equiv1$.

As mentioned above, if $[x]\in\mathcal E(M,\overline M)$ is a critical point of $\mathfrak f_H$, then the second variation of $\mathfrak f_H$ at $[x]$ is identified with the quadratic form \eqref{eq:secondvar} on $T_{[x]}\mathcal E(M,\overline M)\cong\mathcal C^{2,\alpha}(M)$.
The differential $\mathrm d(\mathfrak d\mathfrak f_H)(x_{r_*})\colon B_2\to B_0$ is the linearization of the mean curvature function, which is precisely the negative Jacobi operator $-J_{x_{H_{r_*}}}$. This is an essentially positive Fredholm operator from $\mathcal C^{2,\alpha}(M)$ to $\mathcal C^{0,\alpha}(M)$, see \cite{Whi, Whi2}.\footnote{Indeed, observe that $-\Delta_{x_{H_r}}$ is a positive isomorphism from $\mathcal C^{2,\alpha}(M)$ to $\mathcal C^{0,\alpha}(M)$.} Thus, assumption (a) of Theorem~\ref{thm:Gbifurcation} is satisfied, see Remark~\ref{thm:fredhassumption}. Assumptions (b1), (b2) and (b3) respectively imply the hypotheses (b),\footnote{The group
$G^0_r$ is \emph{nice} in the sense of \cite{SmoWas} because it is connected.} (c) and (d) of Theorem~\ref{thm:Gbifurcation}.
The claimed result then follows immediately from Theorem~\ref{thm:Gbifurcation}.
\end{proof}

\begin{rem}
Theorem~\ref{thm:MorseCMCbif} uses the assumption that the mean curvature function $r\mapsto H_r$ has non vanishing derivative at the bifurcation instant $r_*$. Such assumption is used in the proof in order to parameterize the trivial branch of CMC immersions through the value of their mean curvature. A natural question is if this assumption is necessary. The following simple examples show that it is indeed necessary, i.e., bifurcation may not occur otherwise.
\end{rem}

\begin{example}\label{exa:2dimexa}
Consider the two-variable function $f(x,y)=4y^3+6xy^2+3xy-3x^2y$ on the plane. We can regard it as a family of functions of $y$, parameterized by $x$. For each fixed $x$, we look at the critical points of the function $y\mapsto f(x,y)$, i.e., we look for the zeroes of the partial derivative $\frac{\partial f}{\partial y}=12y^2+12xy-3x+3x^2$.
Near $(0,0)$, the points $(x,y)$ that solve $\frac{\partial f}{\partial y} =0$ form a smooth curve\footnote{By explicit calculation, the curve is the graph of the function $x=\frac12\left(1-4y-\sqrt{1-8y}\right)$.} contained in the half-plane $x\ge0$, tangent to the $y$ axis at $(0,0)$. Notice that the Implicit Function Theorem cannot be used in this situation, as
$\frac{\partial^2f}{\partial y^2}(0,0)=0$. Observe also that the function $x$ is not locally injective on the points of the curve near $(0,0)$, since for each $x>0$ there are exactly two solutions of $12y^2+12xy-3x+3x^2=0$, one with $y>0$ and another with $y<0$. At all points $(x,y)$ on this curve with $y>0$, the second derivative $\frac{\partial^2f}{\partial y^2}=24y+12x$ is positive, while it is negative at all points $(x,y)$ on the curve with $y<0$. Thus, there is a \emph{jump of the Morse} index at the point $(0,0)$, but there is \emph{no bifurcation}.
\end{example}

\begin{example}\label{ex:spheres}
An explicit counter-example to CMC bifurcation can be given when assumption {\rm (a)} of Theorem~\ref{thm:MorseCMCbif} is not satisfied. Consider the family $[-1,1]\ni r\mapsto x_r$, where $x_r$ is the embedding into $\mathds R^3$ of the spherical cap above the $xy$-plane of the round sphere centered at $(0,0,r)$ of radius $\sqrt{1+r^2}$. These spherical caps have the same boundary, which is the circle $C$ of radius $1$ in the $xy$-plane centered at the origin, see Figure~\ref{fig:threespheres}.
\begin{figure}[htf]
\vspace{-2.5em}
\includegraphics[scale=.35]{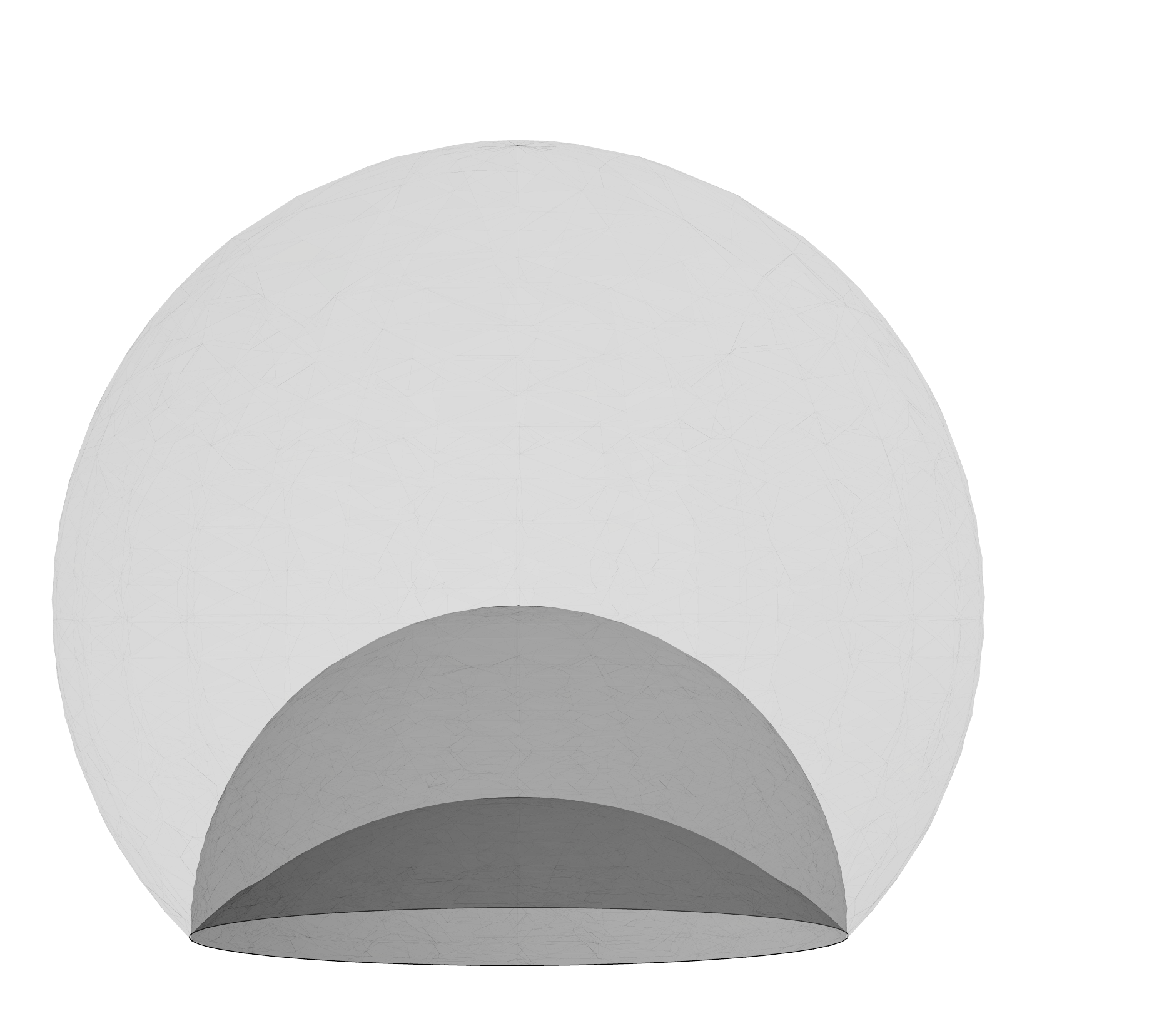}
\caption{Family of spherical caps with the same boundary $C$, the unit circle in the $xy$-plane.}
\label{fig:threespheres}
\end{figure}
Both principal curvatures of $x_r$ are equal to $\frac{1}{\sqrt{1+r^2}}$, hence also its mean curvature is $H_r=\frac{1}{\sqrt{1+r^2}}$. Notice that $H_r$ attains its maximum $H_0=1$ at the half-sphere, hence assumption {\rm (a)} of Theorem~\ref{thm:MorseCMCbif} is \emph{not} satisfied when $r_*=0$.

All other assumptions {\rm (b1)}, {\rm (b2)} and {\rm (b3)} are satisfied. Namely, the only degeneracy instant\footnote{Note that when $r=0$, there exists a Jacobi field $f=\langle K,N_{x_0}\rangle$, where $K=\frac{\partial}{\partial z}$ and $N_{x_0}$ is the unit normal field along $x_0$. This Jacobi field is in $\mathrm{Jac}_{x_0}$ but not in $\mathrm{Jac}^K_{x_0}$, since $K$ is not tangent to the half-sphere (but only to its boundary). Hence, $x_0$ is a \emph{degenerate} CMC embedding.} of $(x_r)_{r}$ is precisely $r_*=0$. A jump of the Morse index can be obtained applying an adequate version of the Morse Index Theorem to $(x_r)_r$. In fact, $\mathrm i_{\mathrm{Morse}}(x_r)$ can be written as the sum of degeneracy instants $s\in [-1,r]$ (counted with multiplicity), and hence is a non-decreasing function of $r$ that jumps as $r$ crosses $r_*=0$.

Finally, bifurcation \emph{does not happen} at $r_*=0$. Since $(x_r)_r$ are embedded in the half-space $z>0$ of $\mathds R^3$ and meet the plane $z=0$ transversely, along the circle $C$, any bifurcating branch would satisfy the same properties for a short time. From the a maximum principle type argument (the Alexander reflection method), any such CMC surfaces must be spherical caps, see~\cite{ebmr}.
\end{example}

\subsection{Equivariant bifurcation using representations}
It is possible to use representation theory to prove a slight generalization of Theorem~\ref{thm:MorseCMCbif}, that gives a subtler criterion for equivariant bifurcation, without necessarily having a jump of the Morse index. As we will see in Subsection~\ref{subsec:rotsym}, this finer result is efficient in geometric applications where the direct computation of the Morse index is not feasible.

As mentioned above, given a transversely oriented CMC embedding $x\colon M\to\overline M$, we identify the tangent space $T_{[x]}\mathcal E(M,\overline M)$ (i.e., the space of normal vector fields along $x$) with the Banach space $\mathcal C^{2,\alpha}(M)$. Under this identification, we may consider the isotropy representation at $x$, induced by the left-composition action of $\Iso(\overline M,\overline g)$, as the representation $\pi\colon G_x\to\mathrm{GL}\big(T_{[x]}\mathcal E(M,\overline M)\big)$ that maps $\psi\in G_x$ to the operator of left-composition with $\dd\psi$, i.e.,
\begin{eqnarray*}
\pi\colon G_x\times T_{[x]}\mathcal E(M,\overline M) &\longrightarrow& T_{[x]}\mathcal E(M,\overline M) \\
(\psi,f)&\longmapsto & \dd\psi\circ f
\end{eqnarray*}
In more elementary terms, $\pi(\psi)$ acts as follows on a normal vector field $f\in\mathcal C^{2,\alpha}(M)$ along $x$. Consider the variation of $x$ induced by $f$, $x_s=\exp^\perp(sfN_x)$, $s\in\, ]-\varepsilon,\varepsilon[$. Then $\pi(\psi)f$ is the normal vector field $\left.\frac{\dd}{\dd s}\psi\circ x_s\right|_{s=0}$ along $x$.

If $f\colon M\to\mathds R$ is an eigenfunction of the Jacobi operator $J_x$, see \eqref{eq:JacobioperatorCMC}, then $\pi(\psi)f=\dd\psi\circ f$ is another eigenfunction with the same eigenvalue, for all $\psi\in\mathrm{Iso}(\overline M,\overline g)$. This means that the isotropy representation $\pi$ of $G_x$ restricts to a representation of $G_x$ on each eigenspace of the Jacobi operator. More precisely, if $\lambda$ is in the spectrum $\sigma(J_x)$ of $J_x$ and $E_x^\lambda$ is the corresponding eigenspace, we let $\pi_x^\lambda\colon G_x\to\mathrm{GL}(E_x^\lambda)$ be the representation defined as the restriction of $\pi$ to $E_x^\lambda$, i.e.,
\[\phantom{,\quad\psi\in G_x,\ f\in E_x^\lambda.}\pi_x^\lambda(\psi)f=\dd\psi\circ f,\quad\psi\in G_x,\ f\in E_x^\lambda.\]
Let us define the \emph{negative isotropy representation} $\pi_x^-$ as the direct sum of all the representations $\pi_x^\lambda$,
as $\lambda$ varies in the set of negative eigenvalues of $J_x$, i.e.,
\begin{equation*}
\pi_x^-=\bigoplus_{\substack{\lambda\in\sigma(J_x) \\ \lambda<0}}\pi_x^\lambda,
\end{equation*}
which is a representation of $G_x$ on $E_x^-=\bigoplus\limits_{\substack{\lambda\in\sigma(J_x) \\ \lambda<0}}E_x^\lambda$, compare to \eqref{eq:negisorep}.

Observe that $\mathrm i_{\mathrm{Morse}}(x)=\mathrm{dim}(E_x^-)$. In Theorem~\ref{thm:MorseCMCbif}, assumption (b3) can be hence written as $\dim(E_{x_{r_*-\varepsilon}}^-)\neq\dim(E_{x_{r_*+\varepsilon}}^-)$. Recall that two representations $\pi_i\colon H\to\mathrm{GL}(V_i)$, $i=1,2$, of $H$ are \emph{equivalent} if there exists an $H$-equivariant isomorphism from $V_1$ to $V_2$, which in particular implies $\dim V_1=\dim V_2$. With this notion, we can weaken (b3) and still obtain bifurcation.

\begin{teo}\label{thm:equivCMCbifurcation}
Replace the assumption {\rm (b3)} of Theorem~\ref{thm:MorseCMCbif} with
\begin{itemize}
\item[(b3')]  the representations $\pi_{x_{r_*-\varepsilon}}^-$ and $\pi_{x_{r_*-\varepsilon}}^-$ of $G_r^0$ are not equivalent.
\end{itemize}
Then, the same conclusion holds, i.e., $r_*$ is a bifurcation instant for the family $(x_r)_r$.
\end{teo}

\begin{proof}
The same proof of Theorem~\ref{thm:MorseCMCbif} applies, where instead of Theorem~\ref{thm:Gbifurcation}, we use Theorem~\ref{thm:Gbifurcation2} to obtain the conclusion.
\end{proof}

\begin{rem}\label{thm:remboundary}
Theorems~\ref{thm:MorseCMCbif} and \ref{thm:equivCMCbifurcation} hold \emph{verbatim}
in the case of CMC embeddings of compact manifolds $M$ with boundary.
In this case, a fixed boundary condition is necessary, namely, assume that the embeddings $x_r$
satisfy $x_r(\partial M)=\Sigma$, with $\Sigma$ a fixed codimension $2$ submanifold of $\overline M$.
In this situation, the notion of nondegeneracy requires a suitable modification.
Given a CMC embedding $x\colon M\to\overline M$ satisfying $x(\partial M)=\Sigma$,
the space $\mathrm{Jac}_x$ is the set of Jacobi fields along $x$ that vanish on $\partial M$, and
the space $\mathrm{Jac}^K_x$ is defined to be the vector space spanned by all functions of the form $\overline g(K,N_x)$,
where $K$ is a Killing vector field in $(\overline M,\overline g)$ that is tangent to $x(M)$ along $x(\partial M)$.
Then, $x$ is said to be \emph{nondegenerate} if $\mathrm{Jac}^K_x=\mathrm{Jac}_x$.
\end{rem}

\begin{rem}\label{thm:remvariablegroup}
As we saw in Example~\ref{exa:2dimexa}, assumption (a) cannot be omitted in Theorems~\ref{thm:MorseCMCbif} and \ref{thm:equivCMCbifurcation}. However, it seems reasonable that assumption (b2) may be weakened. Consider the more general case in which the identity connected component $G^0_r$ of the isotropy of $x_r$ is a \emph{continuous family of Lie groups}. This means that the set $\bigcup_{r\in[a,b]}G^0_r$ has the structure of a \emph{topological groupoid} over the base $[a,b]$, with source and range map given by the projection onto $[a,b]$. A notion of continuity (in fact, smoothness)
for families of Lie groups is given in \cite{TasUmeYam}, and a CMC version of an equivariant implicit function theorem in the case of varying isotropies is discussed in \cite{UmeYam}, where the authors prove the existence of non embedded CMC tori in spheres and hyperbolic spaces. Evidently, the validity of an equivariant bifurcation result in the case of varying isotropies would employ a theory of existence of slices for groupoid affine actions, along the lines of the results discussed in Section~\ref{sec:slices}. This is a topic of current research by the authors.
\end{rem}

\subsection{Clifford tori in round and Berger spheres}
Let us discuss some bifurcation results for CMC hypersurfaces by the second named author and others that can be reobtained as an application of Theorem~\ref{thm:MorseCMCbif}.

The family $x_r\colon \mathds S^n\times\mathds S^m\to\mathds S^{n+m+1}$ of CMC Clifford tori in the round sphere, defined by
\begin{equation}\label{eq:CMCClifford}
x_r(x,y)=\left(r\, x,\sqrt{1-r^2}\, y\right), \quad r\in\left]0,1\right[,
\end{equation}
is studied in~\cite{AliPic11}. The central result gives the existence of two sequences $r_n$ and $r_n'$, with $\lim_{n\to\infty}r_n=0$ and $\lim_{n\to\infty}r_n'=1$, of degeneracy instants for the embeddings $x_r$, with bifurcation at each such instant. In the case $n=m=1$, an explicit description of the bifurcating branches is given in \cite{HyndParkMcCuan}; such branches are formed by rotationally symmetric embeddings of $\mathds S^1\times\mathds S^1$ that are analogue to the classical \emph{unduloids} in $\mathds R^3$. The connected component of the identity of the isotropy of every Clifford torus $x_r$ is the group $\mathrm{SO}(n+1)\times\mathrm{SO}(m+1)$, diagonally embedded into the isometry group $\mathrm{SO}(n+m+2)$ of the round sphere $\mathds S^{n+m+1}$. The Jacobi operator of Clifford tori has a simple form, due to the fact that the Ricci curvature of the ambient and also the norm of the second fundamental form are constant functions. Moreover, the induced metric is the standard product metric of $\mathds S^n\times\mathds S^m$. Nondegeneracy and jumps of the Morse index are computed explicitly in this situation using the eigenfunctions of the Laplacian on $\mathds S^n\times\mathds S^m$.

A similar analysis is carried out in \cite{LimLimPic2011}, in the case of embeddings $x_{r,\tau}\colon\mathds S^1\times\mathds S^1\to\mathds S^3_\tau$, as in \eqref{eq:CMCClifford}, into the $3$-dimensional Berger sphere $\mathds S^3_\tau$, with $r\in\left]0,1\right[$ and $\tau>0$.
In analogy with the standard round case (which corresponds to $\tau=1$), also these embeddings are called \emph{Clifford tori}, and
they have constant mean curvature. For $\tau\ne1$, the identity connected component of the isometry group of $\mathds S^3_\tau$ is $\mathrm{SU}(2)$, and the isotropy of every Clifford torus $x_{r,\tau}$ is $\mathds S^1\times\mathds S^1$, diagonally
embedded in $\mathrm{SU}(2)$. The space of Killing-Jacobi fields along $x_{r,\tau}$ has dimension $3$ when $\tau\ne1$, while the dimension is $4$ in the round case $\tau=1$. The induced metric on the torus is flat, but not equal to the product metric.
A spectral analysis of the Jacobi operator, which is the sum of a multiple of the identity and the Laplacian of a flat (but not product)
metric on the torus, is carried out in \cite{LimLimPic2011}, leading to the following bifurcation result.

\begin{teo}\label{thm:mainBerger}
For every $\tau>0$, there exists a countable set $\mathcal A_\tau\subset\left]0,1\right[$ with the following properties:
\begin{itemize}
\item[(a)] $\inf\mathcal A_\tau=0$ and $\sup\mathcal A_\tau=1$;
\item[(b)] for all $r_*\in\mathcal A_\tau$, the family $r\mapsto x_{r,\tau}$ \emph{bifurcates} at $r=r_*$
\end{itemize}
Furthermore, for every $r\in\left]0,1\right[$ there exists a countable
set $\mathcal B_r\subset\left]0,1\right[\bigcup\left]1,+\infty\right[$  with the following properties:
\begin{itemize}
\item[(c)] $\sup\mathcal B_r=+\infty$;
\item[(d)] given $r\in\left]0,1\right[$, for all $\tau_*\in\mathcal B_r$ the family $\tau\mapsto x_{r,\tau}$
bifurcates at $\tau=\tau_*$.
\end{itemize}
\end{teo}
The above, as well as the bifurcation statement in the case of the round sphere, can be proved as an application of Theorem~\ref{thm:MorseCMCbif}.

\subsection{Rotationally symmetric surfaces in $\mathds R^3$}\label{subsec:rotsym}
Both results discussed above of bifurcation for the families of Clifford tori in round and Berger spheres can be obtained as an application of Theorem~\ref{thm:MorseCMCbif}, given that there is a jump of the Morse index at every degeneracy instant. However, an explicit computation of the Morse index is not feasible in many situations, whereas the weaker assumption of Theorem~\ref{thm:equivCMCbifurcation} on the \emph{jump} of the isotropy representation may actually be an easier task. An example of this situation is provided by rotationally symmetric CMC surfaces in $\mathds R^3$. This problem is studied in detail in~\cite{KoiPalPic2011}.
\begin{figure}[htf]
\includegraphics[scale=.305]{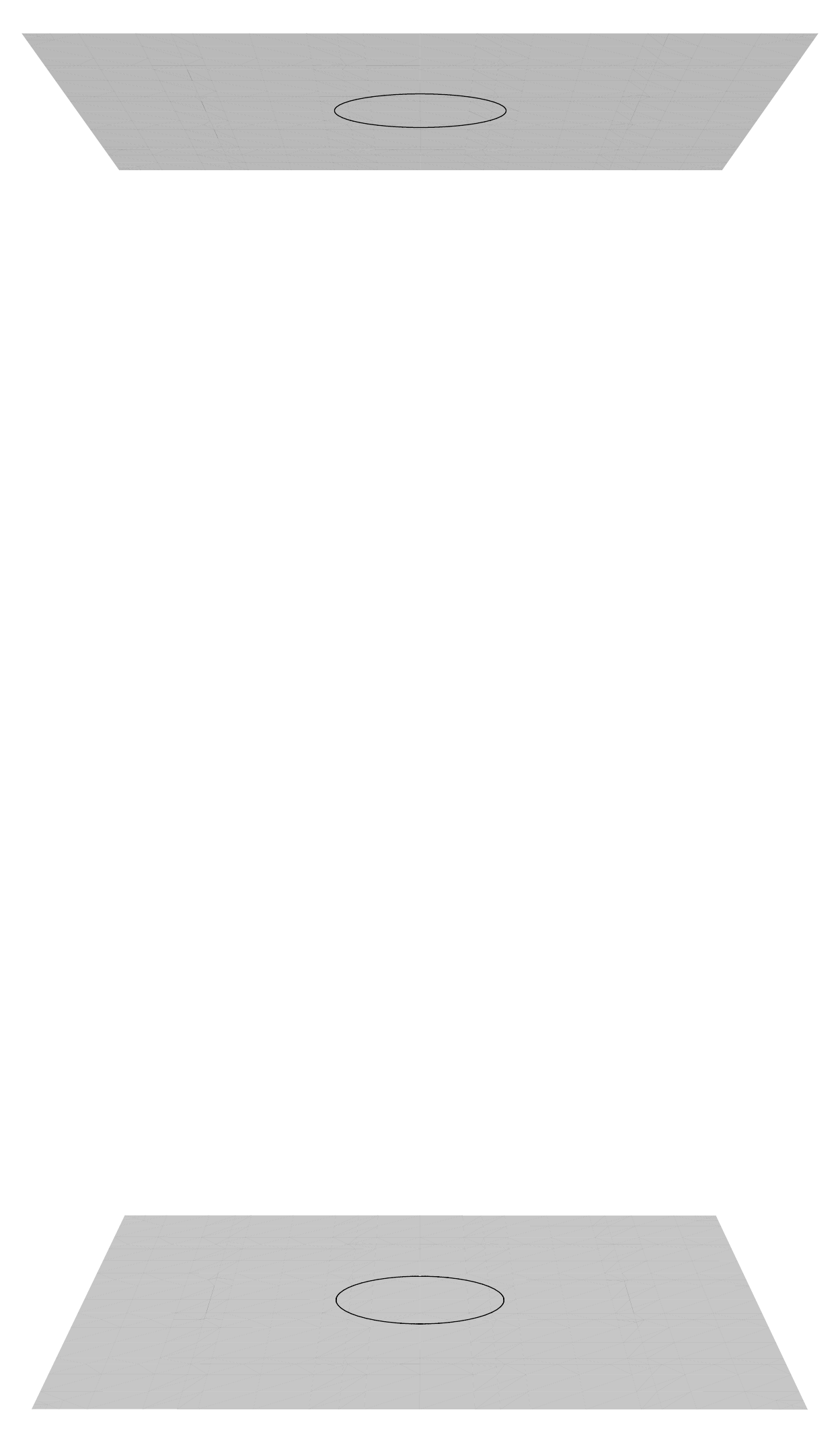} \includegraphics[scale=.3]{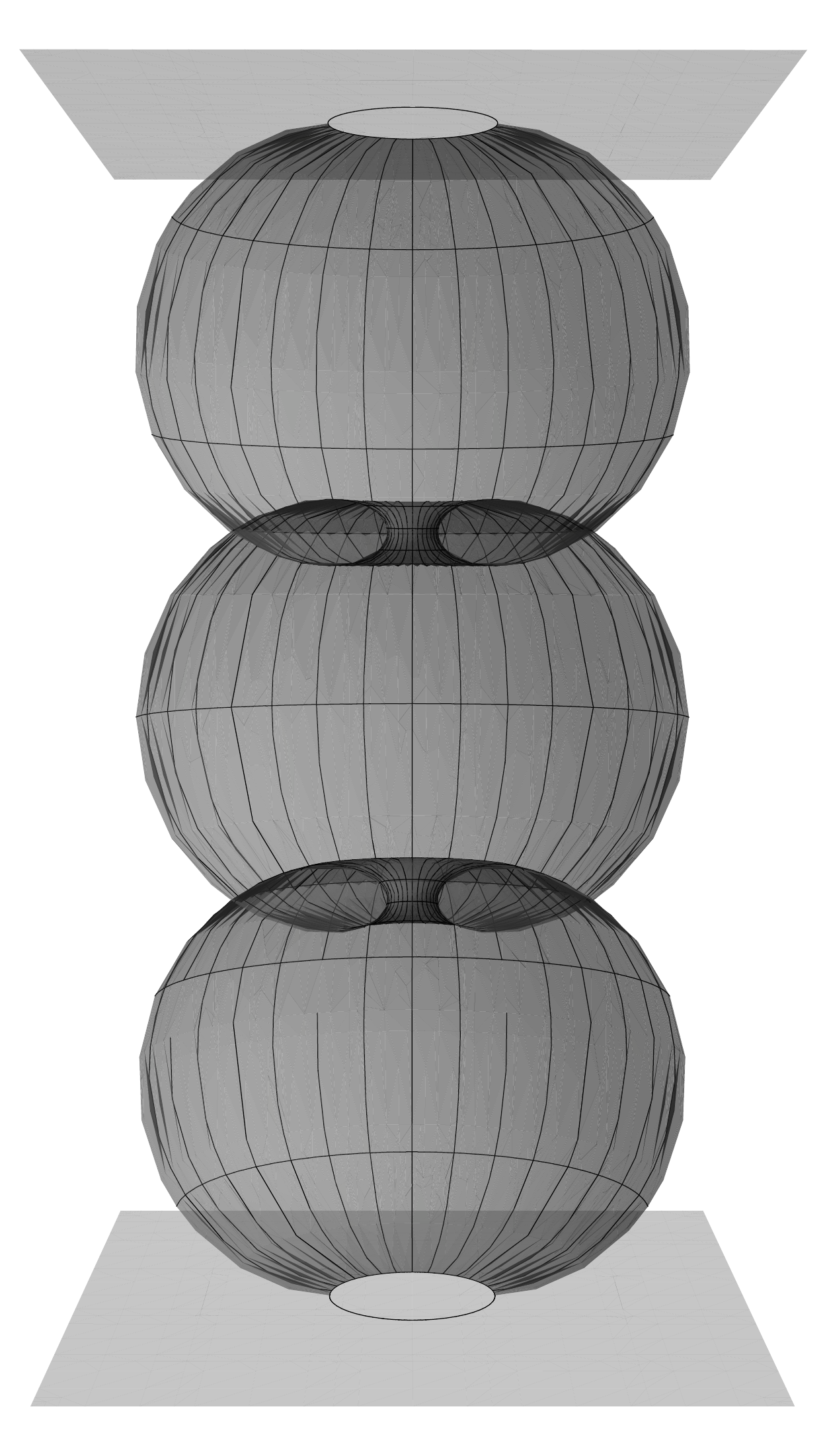} \includegraphics[scale=.3]{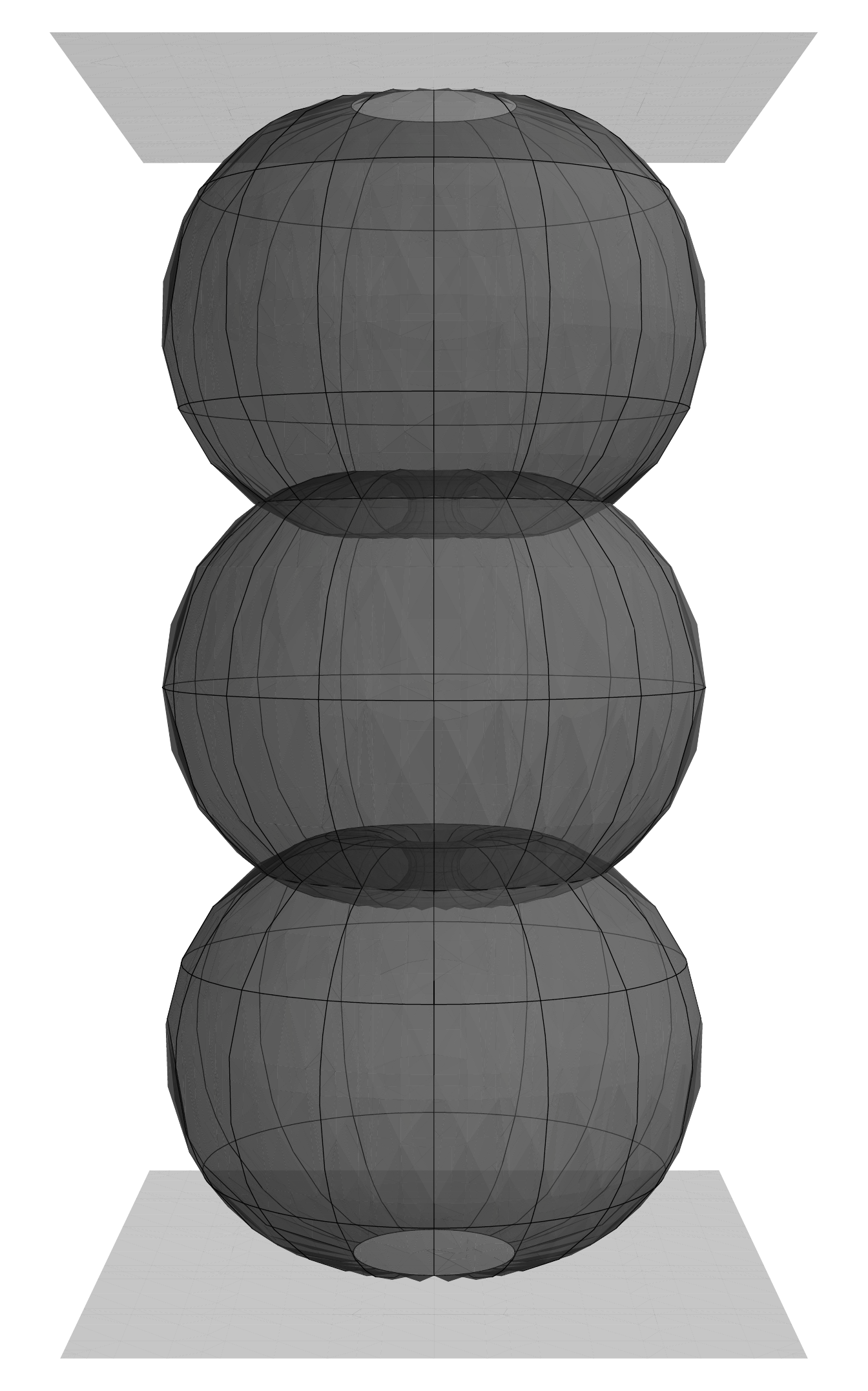}
\caption{The boundary conditions considered for rotationally symmetric CMC surfaces in $\mathds R^3$, and an example of such a surface, a \emph{nodoid} (viewed in half in the second picture and full in the third picture).}
\label{fig:cmcrot}
\end{figure}

For convenience of notation, write $\mathds S^1=[0,2\pi]/\{0,2\pi\}$. Let us consider the case of a family of fixed boundary CMC rotationally symmetric surfaces $\mathrm x_r\colon [0,L_r]\times\mathds S^1\to\mathds R^3$, $r\in I\subset\mathds R$, whose boundary in the union of two co-axial circles lying in parallel planes of type $z=const.$, see Figure~\ref{fig:cmcrot}. Assuming that the rotation axis is the line $x=y=0$, then $\mathrm x_r(s,\theta)$ can be parameterized by
\[x(s)=x_r(s)\cos\theta,\quad y(s)=x_r(s)\sin\theta,\quad z(s)=z_r(s),\]
for some smooth functions $x_r>0$ and $z_r$, where $s\in[0,L_r]$ is the arc-length parameter of the plane curve $\gamma_r(s)=\big(x_r(s),z_r(s)\big)$, and $\theta\in\mathds S^1$. A direct computation gives that the Laplacian of the induced Riemannian metric on the cylinder $M=[0,L_r]\times\mathds S^1$ is
\[\Delta_r=\frac1{x_r}\,\frac\partial{\partial s}\left(x_r\frac\partial{\partial s}\right)+\frac1{x_r^2}\,\frac{\partial^2}{\partial\theta^2},\]
and the square norm of the second fundamental form is
\[\Vert A_{\mathrm x_r}\Vert^2=(\dot x_r\ddot z_r-\ddot x_r\dot z_r)^2+\frac{\dot z_r^2}{x_r^2},\]
where dot represents derivative with respect to $s$. The eigenvalue problem for the Jacobi equation reads
\[J_r(F)=-\Delta_rF-\Vert A_{\mathrm x_r}\Vert^2\,F=\lambda\,F,\quad F(0,\theta)=F(L_r,\theta)\equiv0.\]
Separation of variables $F(s,\theta)=S(s)T(\theta)$ yields the following pair of boundary value problems for ODE's:
\begin{eqnarray*}
&T''+\kappa T=0, & T(0)=T(2\pi),\; T'(0)=T'(2\pi), \\
&-(x_rS')'+\left(\frac\kappa{x_r}-x_r\Vert A_{\mathrm x_r}\Vert^2\right)S=\lambda\, x_rS, & S(0)=S(L_r)=0.
\end{eqnarray*}
The first problem has non trivial solutions when $\kappa=n^2$, $n\in\mathds Z$, $n\geq 0$, with corresponding
eigenfunctions $\cos n\theta$ and $\sin n\theta$; substituting $\kappa=n^2$ in the second problem we get:
\begin{equation}\label{eq:SL}
\left\{\begin{aligned}-&(x_rS')'+\left(\frac{n^2}{x_r}-x_r\Vert A_{\mathrm x_r}\Vert^2\right)S=\lambda x_rS,\\& S(0)=S(L_r)=0.\end{aligned}\right.
\end{equation}
Every (nontrivial) solution $S_{r,n}$ of the Sturm-Liouville system \eqref{eq:SL} produces two (nontrivial) eigenfunctions of the Jacobi operator along the CMC surface $\mathrm x_r$, given by $S_{r,n}\cos n\theta$ and $S_{r,n}\cos n\theta$. The classical Sturm-Liouville theory gives the existence of an unbounded sequence of eigenvalues of \eqref{eq:SL}, and the corresponding eigenfunctions are smooth and form a Hilbert basis of $L^2([0,L_r])$. Every solution of the above system with $n>0$ produces eigenfunctions that are \emph{not} rotationally symmetric. The rotationally symmetric solutions correspond to $n=0$, in which case the Sturm-Liouville equation reads:
\begin{equation}\label{eq:SLn=0}
\left\{\begin{aligned}-&(x_rS')'-x_r\Vert A_{\mathrm x_r}\Vert^2\,S=\lambda x_rS,\\& S(0)=S(L_r)=0.\end{aligned}\right.
\end{equation}

We will say that $r$ is a \emph{conjugate instant} for the Sturm-Liouville problem if \eqref{eq:SL}, has a non trivial solution with $\lambda=0$. Evidently, if $r$ is a conjugate instant, then the CMC embedding $\mathrm x_r$ is degenerate. In this setting, Theorem~\ref{thm:equivCMCbifurcation} can be applied to obtain the following bifurcation result.

\begin{teo}\label{thm:bifrotCMCsurfaces}
Consider the family $r\mapsto\mathrm x_r$ of rotationally symmetric CMC surfaces in $\mathds R^3$ having fixed boundary described above. For every fixed $n>0$, let $r_n$ be the first conjugate instant of the Sturm-Liouville equation \eqref{eq:SL}. Assume that $r_n$ is an isolated degeneracy instant for the family $\mathrm x_r$, and that the derivative of the mean curvature function $H'_{r_n}$ is not zero. Then, $r_n$ is a bifurcation instant for the family of CMC surfaces $(\mathrm x_r)_r$. Moreover, if $r_n$ is not a conjugate instant also for the Sturm-Liouville problem \eqref{eq:SLn=0}, then \emph{break of symmetry} occurs at the bifurcating branch, i.e., the bifurcating branch consists of fixed boundary CMC surfaces that are \emph{not} rotationally symmetric.
\end{teo}

\begin{proof}
Theorem~\ref{thm:equivCMCbifurcation} applies here in the following setup. The (identity connected component of the) isotropy of $\mathrm x_r$ is\footnote{Namely, the subgroup of isometries of $\mathds R^3$ that preserve two co-axial circles lying in parallel planes is generated by the group of rotation around the axis, and, if the two circles have same radius, by the reflection about the plane in the middle of the two parallel planes.}
the group of rotations around the $z$ axis. Assumptions~(a), (b1) and (b2) of Theorem~\ref{thm:MorseCMCbif} hold at the instant $r_n$ under our hypotheses. Assumption~(b3') of Theorem~\ref{thm:equivCMCbifurcation} holds at the first instant at which \eqref{eq:SL} admits a non trivial solution. Namely, for $r<r_n$, the negative isotropy representation $\pi^-_{\mathrm x_r}$ of the group of rotations has no vector whose isotropy is isomorphic to $\mathds Z_{n}$. On the other hand, for $r>r_n$, with $r-r_n$ sufficiently small, the two Jacobi fields determined by the non trivial solution of \eqref{eq:SL} belong to the negative eigenspace of $J_{r}$, and they have isotropy isomorphic to $\mathds Z_{n}$. This implies that for $\varepsilon>0$ small enough, the representations $\pi^-_{\mathrm x_{r_n-\varepsilon}}$ and $\pi^-_{\mathrm x_{r_n+\varepsilon}}$ are not equivalent. Thus, from Theorem~\ref{thm:equivCMCbifurcation}, bifurcation occurs at $r_n$. As to the break of symmetry, we observe that if $r_n$ is not a conjugate instant for \eqref{eq:SLn=0}, then the symmetrized CMC variational problem is nondegenerate at $r_n$, and bifurcation by rotationally symmetric CMC embeddings cannot occur.
\end{proof}

We observe that, under the assumptions of Theorem~\ref{thm:bifrotCMCsurfaces}, jump of the Morse index may \emph{not} occur at $r_n$. An example where the above result applies is provided by families of fixed boundary \emph{nodoids}, see~\cite{KoiPalPic2011} and Figure~\ref{fig:cmcrot}.
\end{section}

\appendix
\begin{section}{Nonlinear formulation of the bifurcation result}
\label{app:A}

We consider a variational setup similar to that of Section~\ref{sec:slices}, namely $\mathcal M$ is a smooth Banach manifold, $G$ is a Lie group acting continuously by diffeomorphisms on $\mathcal M$ (recall the auxiliary maps \eqref{eq:defbetax}), and we also have
\begin{itemize}
\item[(a)] $\mathcal E\to\mathcal M$ is a Banach vector bundle over $\mathcal M$;
\item[(b)] $[a,b]\ni\lambda\mapsto T_\lambda\in\Gamma(\mathcal E)$ is a continuous path of $G$-equivariant sections;\item[(c)] the action of $G$ on $\mathcal M$ \emph{lifts} to an action of $G$ on $\mathcal E$, which is linear on the fibers;
\item[(d)] $[a,b]\ni\lambda\mapsto x_\lambda\in\mathcal M$ is a continuous path such that $T_\lambda(x_\lambda)=0$, for all $\lambda$.
\end{itemize}

Analogously to Definition~\ref{def:eqbif}, an instant $\lambda_*\in[a,b]$ is an \emph{equivariant bifurcation instant for the family $(T_\lambda,x_\lambda)_{\lambda\in[a,b]}$} if there is a sequence $(x_n,\lambda_n)\in\mathcal M\times [a,b]$ satisfying \eqref{itm:eqbif1}, \eqref{itm:eqbif3} and $T_{\lambda_n}(x_n)=0$, for all $n$, which corresponds to \eqref{itm:eqbif2}. In order to give an existence result for an equivariant bifurcation instant, let us consider the following auxiliary\footnote{Compare with the structure employed in \cite[Section 3]{BPS1}.} structure
\begin{itemize}
\item[(e)] $\mathfrak i\colon T\mathcal M\to\mathcal E$ is a $G$-equivariant continuous inclusion (i.e., an injective morphism of vector bundles) with dense image;
\item[(f)] $\langle\cdot,\cdot\rangle$ is a $G$-invariant continuous (but not necessarily complete) positive-definite inner
product in the fibers of $\mathcal E$;
\item[(g)] $\mathfrak j_\lambda\colon\mathcal E_{x_\lambda}\to T_{x_\lambda}\mathcal M^*$ is the map $\mathfrak j_\lambda(e)v=\langle e,\mathfrak i(v)\rangle$, and the composition $\mathfrak j_\lambda\circ(\mathrm d^\mathrm{ver}T_\lambda)(x_\lambda)\colon T_{x_\lambda}\mathcal M\to T_{x_\lambda}\mathcal M^*$ is symmetric for all $\lambda$.
\end{itemize}

For all $\lambda\in[a,b]$ and all $\eta\ge0$, set
\[N_{\lambda,\eta}:=\mathrm{span}\big\{v\in T_{x_\lambda}\mathcal M:\mathrm d^{\mathrm{ver}}T_\lambda(x_\lambda)v=\mu\,\mathfrak i(v),\ \mu\leq\eta\big\},\]
and $N_\lambda:=N_{\lambda,0}$, compare with \eqref{eq:nlambdae} and \eqref{eq:nlambda}. If $G_\lambda\subset G$ is the isotropy of $x_\lambda$, we have the isotropy representation $G_\lambda\ni g\mapsto\mathrm d\phi_g(x_\lambda)\in\mathrm{GL}(T_{x_\lambda}\mathcal M)$. For all $\eta\ge0$, the space $N_{\lambda,\eta}$ is invariant by this action. Denote by $\pi_\lambda^-\colon G_\lambda\to\mathrm{GL}(N_\lambda)$ the restriction of such representation, which is called the \emph{negative isotropy representation} of $G_\lambda$ (compare with \eqref{eq:negisorep}).

We can now state the nonlinear formulation of the celebrated result of J. Smoller and A. Wasserman~\cite[Thm 3.3]{SmoWas}, whose proof follows its linear version, using the above auxiliary structure.

\begin{prop}\label{thm:nonlinSmWas}
In the above setup, assume that
\begin{itemize}
\item[(a)] there exists $\varepsilon>0$ such that $\mathrm{dim}(N_{\lambda,\varepsilon})<+\infty$, for all $\lambda\in[a,b]$;
\item[(b)] for all $\lambda$, $G_\lambda=G$;
\item[(c)] $\mathrm d^\mathrm{ver}T_a(x_a)\colon T_{x_a}\mathcal M\to\mathcal E_{x_a}$ and $\mathrm d^\mathrm{ver}T_b(x_b)\colon T_{x_b}\mathcal M\to\mathcal E_{x_b}$ are isomorphisms;
\item[(d)] the negative isotropy representations $\pi_a^-$ and $\pi_b^-$ are not equivalent.
\end{itemize}
Then, there is an equivariant bifurcation instant in $\left]a,b\right[$ for the family $(T_\lambda,x_\lambda)_{\lambda}$.
\end{prop}

\end{section}


\end{document}